\documentclass[a4paper,reqno]{paper}
\usepackage[utf8]{inputenc}
\usepackage[T1]{fontenc}
\usepackage{amsfonts,enumerate}
\usepackage{amsmath}
\usepackage{amssymb}
\usepackage{amsthm}
\usepackage{graphics}
\usepackage[backref = page]{hyperref}
\hypersetup{
colorlinks =  true, 
linkcolor = purple!80, 
citecolor =teal
}

\usepackage[capitalize]{cleveref}
\usepackage{tikz}
\usepackage{tikz-cd}

\usepackage{geometry}
\usepackage{amsfonts}
\usepackage{graphicx}
\usepackage{genyoungtabtikz}
\usepackage{dsfont}
\usepackage{mathrsfs}
\usepackage{ytableau}
\usepackage{stmaryrd}

\usepackage{aliascnt}
\usepackage{mynewcommands}

\usepackage{titletoc,tocloft}
\let\svthefootnote\thefootnote
\newcommand\freefootnote[1]{%
  \let\thefootnote\relax%
  \footnotetext{#1}%
  \let\thefootnote\svthefootnote%
}

    \definecolor{mygreen}{RGB}{3,160,74}

\newcommand\gry{\Yfillcolour{black!40}}
\newcommand\white{\Yfillcolour{white}}

\begin{document}

\sloppy

\title{Entropy of affine permutations and universality of affine atomic lengths}
 \author{
 Nathan Chapelier-Laget\thanks{Université du Littoral Côte d'Opale.
 Email address: {\tt nathan.chapelier@univ-littoral.fr}
 },
 Thomas Gerber\thanks{Université Lyon 1.
 Email address: {\tt gerber@math.univ-lyon1.fr}
 },
 Nicolas Jacon \thanks{Université Champagne-Ardenne.
 Email address: {\tt nicolas.jacon@univ-reims.fr}
 },
 Cédric Lecouvey\thanks{Université de Tours, Institut Denis-Poisson.
 Email address: {\tt cedric.lecouvey@univ-tours.fr}
 }.
 }
\maketitle
\freefootnote{}


\hrule 

\begin{abstract}
We introduce and study the notion of entropy of affine permutations and prove that it coincides with the atomic length 
associated with the sum of the fundamental weights for a type $A$ affine root system, 
as defined by the first two authors. 
We then establish an analogue of the Granville-Ono theorem by showing that any nonnegative integer can be realised as the entropy of an affine permutation or alternatively, 
as the size of a core multipartition as introduced by the last two authors. 
Our proof uses an additive combinatorics theorem due to Hall on difference sets of permutations modulo $n$. 
More generally, we give a polynomial expression of the atomic length associated with any dominant weight in affine type $A$ and investigate the problem of its universality. 
Beyond type $A$, we are able to prove that the entropy of affine type $C_n$ permutations is universal when $2n+1$ is prime. 
This is achieved by establishing an analogue of Hall's theorem for the hyperoctahedral group based on Alon's combinatorial Nullstellensatz. 
We also propose conjectures generalising the results presented in the paper, each supported by computational evidence.
Finally, we show that in any affine classical type, the problem of the universality of the atomic length simplifies in large rank when the weight considered is conveniently adjusted.
\end{abstract}

\hrule



\section*{Introduction}

\subsubsection*{Universal quadratic forms and beyond}

Representing nonnegative integers by sums of squares is a very classical
problem in number theory dating back at least to the work of Fermat in the 17th century.
Later, Lagrange obtained an iconic result: any nonnegative integer
can be written as the sum of at most four squares. This "four-square theorem" has led to many generalisations. 
For example, Ramanujan has proved that there are exactly $54$ possible quadruples of
integers $a\leq b\leq c\leq d$ such that the quadratic form $ax^{2}+by^{2}+cz^{2}+dt^{2}$ represents every nonnegative integer. 
We refer the reader to \cite{grosswald1985representations} for a gentle introduction to the problem of representations of integers by sums of squares. More generally, a
 quadratic form is said to be universal when it represents
every nonnegative integer. In 1993, Conway stated a famous conjecture claiming
that specific quadratic forms are universal if and only if they represent a simple
list of $9$ integers, the largest of which being $15$. This conjecture,
which reduces the universality problem to a simple computational check, was
proved by Bhargava in 2000.
\medskip

Among the many ways to generalise the universality problem of integral
quadratic forms, the following two will be particularly relevant for the
present paper. 
On the one hand, one can again consider a positive definite integral
quadratic form on $\mathbb{Z}^{n}$, but study its universality only on a
subset of $\mathbb{Z}^{n}$. For example, one can impose that the $n$-tuples of
possible integers do not have repetitions. 
It was proved by Wright in 1933 that
for each integer $s\geq5$,  there exists a largest integer $N(s)$ which is not
expressible as a sum of $s$ distinct non-zero squares (i.e. all integers
larger than or equal to $N(s)$ can be written as the sum of $s$ distinct
squares). We refer to \cite{PaulT1994} for more precise references and
estimations of the bound $N(s)$. 
On the other hand,  one can relax the assumption that the polynomial is a quadratic form to study the universality of general integral polynomials of degree $2$. This problem is significantly more complex; consequently, the following discussion will focus on specific families of such polynomials. We refer
the reader to \cite{ChanOh, chan2015representation} for a review of this question.
\medskip

\medskip

In addition to its interest in number theory, the previous universality problem also
has applications in combinatorics and representation theory. Recall that a
partition $\lambda$ of
size $\left\vert \lambda\right\vert $ is a sequence
$\lambda=(\lambda_{1}\geq\cdots\geq\lambda_{m})$ of nonnegative integers such
that $\left\vert \lambda\right\vert =\lambda_{1}+\cdots+\lambda_{m}$.
A  partition is called an \textit{$n$-core} if it has no hook of length $n$.
For instance, the partition $(10,6,3,3)$ of $22$ is not a $7$-core
since it has a $7$-hook, represented in gray in its Young diagram below. However, one can check that it is a $5$-core.
$$\ds
\Yboxdim{8pt}
\young(<><><><><><><><><><>,<>!\gry<><><><><>,!\white<>!\gry<>!\white<>,<>!\gry<>!\white<>)
$$
When $n=p$ is a prime number, these $n$-cores play an important
role in the representation theory of the symmetric groups in characteristic
$p$ \cite{Brauer1947, Robinson1947}, or, for a general integer $n$, in the representation theory
of the Hecke algebras specialised at
an $n$-th root of unity \cite{JamesMathas1996}. They indeed parametrise the so-called blocks in their
associated decomposition matrices. 
With this in mind, enumerating the $n$-cores becomes an important problem,
and the particular question of finding $n$-cores of a given size is already highly non-trivial.
It turns out that these always exist, as long as $n\geq 4$, and this so-called "$n$-core conjecture" was definitively settled in 1996 by
Granville and Ono \cite{GO1996}.
In their proof, two results due to Garvan, Kim and Stanton \cite{GKS1990} play an essential role: 
the product formula for the generating series of the $n$-core partitions,
as well as the polynomial formula for the size of the $n$-cores, which reads
\begin{equation}\label{GKS_pol}
U(x)=\frac{n}{2}\sum_{i=1}^{n}x_{i}^{2}
+\sum_{i=1}^{n}(i-1)x_{i},
\end{equation}
for $x\in \mathcal{Q} = \{(x_{1},\ldots,x_{n})\in\Z^n \mid x_{1}+\cdots + x_{n}=0\}$,
the subset of $\Z^n$ known to parametrise the $n$-cores.
Therefore, the veracity of the $n$-core conjecture
is equivalent to the universality of $U$, a non homogeneous polynomial of degree two, on $\mathcal{Q}$.
 
\medskip
 
Core partitions also have an interesting interpretation in the representation theory of simple affine Lie algebras. 
These algebras were classified by Kac according to their root system. 
Such a root system allows us to define two important $\mathbb Z$-lattices: the root lattice generated by the simple roots, 
and the weight lattice generated by the fundamental weights. The associated irreducible representations are then labeled by the dominant weights, which are the elements in the cone of nonnegative linear combinations of the fundamental weights. Regarding simple roots, they define affine reflections which generate a Coxeter group $W$ called the Weyl group. 
Both lattices are related by the following crucial property : given a dominant weight $\Lambda$ and an element $w$ in $W$, the difference $\Lambda-w(\Lambda)$ is a linear combination of the simple roots with nonnegative integer coefficients. We refer the reader to \cite{Carter2005} for a complete exposition on affine Lie algebras and their representations. 

\medskip

In \cite{CG2022}, the two first authors define the atomic length of $w$ associated to $\Lambda$  as the sum of the coefficients appearing in the previous decomposition of $\Lambda-w(\Lambda)$. When $\Lambda=\Lambda_0$, the fundamental weight associated to the node $0$ of the Dynkin diagram of type $A_{n-1}^{(1)}$, 
one can easily associate an $n$-core to $w$ so that the atomic length equals its size. 
Hence, the atomic length is the relevant notion to generalise the $n$-core conjecture to this broader, representation-theoretic context. 
It is worth noting that all the aforementioned notions also apply to simple finite-dimensional Lie algebras. 
For $\mathfrak{gl}_n(\C)$, the Weyl group is just the symmetric group on $\{1,\ldots,n\}$, 
and when the dominant weight $\Lambda$ considered is the sum of the fundamental weights, the atomic length of the permutation $w$ coincides with its entropy
\begin{equation}
\label{En}
E(w)=\frac{1}{2}\sum_{i=1}^{n}(w(i)-i)^2,
\end{equation}
which can be regarded as a measure of its distance to the identity. 
It is known that the image of this entropy map is an interval in $\mathbb{N}$ \cite{SackUlfarsson2011}. 
This is also true for any finite root system associated to finite-dimensional simple Lie algebras, as established in \cite{CG2022}.

\medskip

\subsubsection*{Affine entropy, atomic length, and universality}

In this paper, we first explore the notion of entropy for affine permutations (those in the Weyl group of type $A_{n-1}^{(1)}$). 
We prove that it coincides with the atomic length associated to the sum of the fundamental weights of the root system of type $A_{n-1}^{(1)}$ 
and admits a simple polynomial expression similar to its finite counterpart (\ref{En}). We then establish in Theorem \ref{Th_EntropyA} that this entropy (or the associated atomic length) is universal, that is represents all the nonnegative integers if and only if $n\geq5$. To do this, we need an additive combinatorics theorem due to Hall \cite{hall1952combinatorial} on the difference sets of permutations modulo $n$. 

\medskip 

Our second goal is to derive in Proposition \ref{Prop_PolyL} a polynomial expression for the atomic length in type $A_{n-1}^{(1)}$ associated to any dominant weight. This is indeed a crucial step in order to study its universality. This is achieved thanks to results established in \cite{Jacon:size} where such an expression 
is obtain, and gives the number of boxes of $(n,\boldsymbol{s})$-cores, the generalisation of the previous notion of $n$-cores which is relevant in the study of the orbit of a general dominant weight. 
Observe that these $(n,\boldsymbol{s})$-cores also  parametrise the  blocks in decomposition matrices of suitable generalisations of Hecke algebras, the Ariki-Koike algebras (see \cite{JaconLecouvey2021}). 
Hence, universality of the atomic length associated to a dominant weight of type $A_{n-1}^{(1)}$ 
gives the existence of a defect zero block in the associated family of Ariki-Koike algebras.
In particular, the universality of the affine entropy established in \Cref{sec_univA} 
yields a direct analogue of Granville and Ono's classic result on blocks \cite[Corollary 1]{GO1996}: this is \Cref{cor_blocks}.
We note that the polynomial formula for the atomic length associated to any dominant weight also has another interesting application:
 it allows to reduce in Proposition \ref{Prop_DilLattice} the problem of its universality to that of the half-square of the usual Euclidean norm on a simple  subset of $\frac{n}{\ell}\mathbb{Z}^n$ 
(here $\ell$ is the level of the considered weight, that is, 
the sum of its coordinates once expressed on the basis of fundamental weights). 
It is worth pointing out here the analogy with the problem of the decomposition of an integer as a sum of a fixed number of distinct squares evoked earlier: indeed, here again, we are looking for a decomposition in (half) the sum of $n$ squares with restrictive conditions.  
We further conjecture that the atomic length is universal when the dominant weight considered is a sum of fundamental weights corresponding to successive nodes in the Dynkin diagram of type $A_{n-1}^{(1)}$,
see \Cref{conj_trunc_affine}. We also propose a refinement of the Granville-Ono problem in \Cref{conj_refined_GO} based on the combinatorics of $(n,\boldsymbol{s})$-cores.

\medskip 

Finally, our third objective is to explore the universality problem beyond affine type $A$. 
In finite classical types, we prove universality of the corresponding atomic length in \Cref{thm_sat_trunc}, refining the results of \cite{CG2022}
and serving as further motivation for \Cref{conj_trunc_affine}.
We establish in particular that the entropy of affine signed permutations (i.e. elements of the affine Weyl group of type $C_{n}^{(1)}$) also coincides with an atomic length. 
We conjecture that it is yet again universal, and prove this conjecture in Corollary \ref{Cor_UnivB} when $2n+1$ is prime. 
Our proof uses Alon's combinatorial Nullstellensatz \cite{Alon} which enables us to obtain a generalisation of the theorem by Hall used in type $A$ (\Cref{Th_HallB}). We also establish in Theorem \ref{Th_univlarge} the universality of the atomic length in large rank and for a suitable choice of weight.

\medskip

The paper is organised as follows. 
We introduce the notion of entropy for affine permutations in \Cref{sec_entropy} and connect it to the atomic length associated with the sum $\rho$ of the fundamental weights in affine type $A_{n-1}^{(1)}$. In \Cref{sec_univA}, we prove the universality of the entropy on affine permutations. 
\Cref{sec_pol_formulas} is devoted to the polynomial expression of the atomic length for any weight in affine type $A$, its connection with the number of boxes in the $(n,\boldsymbol{s})$-cores and its interpretation in terms of the usual Euclidean norm, culminating in two conjectures.
We deal with all finite classical types in \Cref{sec_finite} for the finite counterpart of our favourite atomic lengths.
In \Cref{sec_univB}, we establish the universality of the entropy on signed permutations (type $C_{n}^{(1)}$). 
Finally, the universality of the atomic length associated with the weight $\rho^{\vee}$ in large rank is proved in \Cref{sec_large_rank}.

\section{Entropy of affine permutations and affine atomic length for the sum of fundamental weights}
\label{sec_entropy}

In \cite{CG2022}, the notion of atomic length was introduced, based on crystal theory and attached to a given dominant weight, 
in order to provide:
\vspace{-0.5em} 
\begin{itemize}
\setlength{\itemsep}{0.01pt} 
\setlength{\parskip}{0.01pt}  
\setlength{\parsep}{0.01pt}
\item a natural generalization of the size statistic on (core) partitions,
\item a natural weighted version of the Coxeter length function on Weyl groups.
\end{itemize}
\vspace{-0.2em} 
Note that in finite type $A$, the $\rho$-atomic length (that is, attached to the half-sum of the positive roots) 
coincides with the (half-)entropy of permutations.
Beyond type $A$, the study of the atomic length has led to a variety of results.
For instance, in \cite{CG2022}, a finite version of the  Granville-Ono theorem for any finite Weyl group was obtained;
a new way to count bigrassmannian elements as well as a refinement of inversion sets was also given; and for
the affine dominant weight $\La_0$, this led
to various results on generalised core partitions \cite{STW2023, LW2024, BCG2024}.

\medskip

The goal of this section is two-fold.
First, we introduce the notion of entropy for an affine permutation and establish that it coincides with the atomic length associated with $\rho$ as defined in \cite{CG2022}. 
Then, we give a simple expression of the entropy in terms of the Euclidean norm, which is more suited to study the problem of its universality.

\subsection{Generalities on the affine atomic length}
\label{subsec_AL_gen}

\medskip

\newcommand{\rank}{n}

Let $W$ be the Weyl group of a rank $\rank$ affine Kac-Moody Lie algebra, for which we have the well-known Dynkin classification \cite{Kac1984}.
We denote $\al_i, i=0,\ldots, \rank$ the real simple roots.
They span a Euclidean vector space denoted by $V$,
and for each $v = \sum_{i=1}^\rank a_i\al_i\in V$, we denote
$\h(v) = \sum_{i=1}^\rank a_i$, the \textit{height} of $v$.
We denote  by $h$ the Coxeter number.
Finally, let $\La_i, 0=1,\ldots, \rank \in V$ be the fundamental weights.

\begin{Def} Let $\La$ be a dominant weight.
The \textit{$\La$-atomic length} is the map $\sL_\La : W \to \N, w \mapsto \h(\La-w\La).$
\end{Def}

Note that $\sL_\La$ indeed takes nonnegative integer values since $\La-w\La$ is a nonnegative sum of positive roots.
A more explicit formula for the atomic length can be given by using the semi-direct decomposition of the affine Weyl group $W=M\rtimes {\mathring{W}}$
where $M$ is the coroot or root lattice (depending on the Dynkin type) and ${\mathring{W}}$ is the finite Weyl group.
For $w\in W$, we denote accordingly $w = t_{x} \overline{w}$ the decomposition as a product of a translation of $M$ 
and a finite Weyl group element.
The following has been established in \cite[Lemma 8.1]{CG2022}.

\begin{equation}\label{formula_AL}
\sL_{\La} (w)  = \sL_{\overline{\La}}(\overline{w}) - \ell \h({x}) + h \left( \langle \overline{\La} , \overline{w}^{-1}({x}) \rangle + \frac12 \Vert{x}\Vert^2 \ell \right)
\end{equation}
where $\overline{\La}$ is the finite part of $\La$ and $\ell$ is the level of $\La$.

\begin{Exa}
Let $\La=\La_0$. Formula \Cref{formula_AL} gives
$$\sL_{\La_0} (w) = \frac{h}{2}\Vert{x}\Vert^2 -\h({x}).$$
Note that this depends only on ${x}$, thus we obtain a statistic on affine Grassmannian elements.
In type $A_{\rank}^{(1)}$, this recovers the polynomial formula for the size of $(\rank +1)$-core partitions
established in \cite{GKS1990} and recalled in Formula \Cref{GKS_pol}.
In other types, it is this statistic that
has been recently studied under various perspectives \cite{STW2023, LW2024, BCG2024}.
\end{Exa}

The dominant weight $\rho=\sum_{i=0}^{\rank}\La_i$ will be of particular interest to us.
In this case, we simply denote $\sL=\sL_\rho$ the corresponding atomic length.
Denote further $\overline{\rho}=\sum_{i=1}^{\rank} \omega_i$
where $\omega_i$ are the fundamental weights for the corresponding finite root system. 
Recall that $\overline{\rho}$ is also the half-sum of the positive real roots.

\subsection{Atomic length for affine permutations using window notation}

Let now $n\geq 2$ and let us focus on type $A_{n-1}^{(1)}$ and its corresponding affine Weyl group $W$.
Classically, we construct the simple roots $\al_i$, for $i=1,\ldots, n-1$, inside $\R^{n}$ by setting $\al_i=\eps_i-\eps_{i+1}$
where $(\eps_i)_{1\leq i\leq n}$ is the standard basis of $\R^n$. It is easy to see that the reflections fixing the hyperplanes orthogonal to these simple roots generate a group ${\mathring{W}}$ isomorphic to the symmetric group on $n$ elements. 
The affine Weyl group $W$ is obtained by adding an affine reflection to this set of generators,
and we have the semidirect decomposition $W=M\rtimes {\mathring{W}}$ where $M=\bigoplus_{i=1}^{n-1}\Z\al_i$ is the root lattice.
The elements of $W$ can be regarded as affine permutations
and  represented by using the so-called \textquotedblleft window
notation\textquotedblright, see for instance \cite[Chapter 2, Section 1]{LLMSSZ2014}.
More precisely, we write
\begin{equation}
w=\left[
\begin{array}
[c]{cccccc}
1 & 2 & \cdots & \cdots & n-1 & n\\
w(1) & w(2) & \cdots & \cdots & w(n-1) & w(n)
\end{array}
\right]  \label{Window}
\end{equation}
where the $w(i)$'s are integers with distinct residues modulo $n$ summing up
to $\frac{n(n+1)}{2}$. An affine permutation satisfies in particular the periodicity property
$$w(i+kn)=w(i)+kn$$
for any $k$ in $\mathbb{Z}$ which permits to compute the image of any integer from the previous window notation.  The corresponding finite Weyl group is ${\mathring{W}}$, the symmetric group on $n$ elements.
It is easy to check that right multiplication
by an element $u\in {\mathring{W}}$ changes the $i$-th
column of $w$ into its $u(i)$-th column.
If follows that the translations $t_x$, $x\in M$, are the element $w\in W$ verifying
$w(i)=i+n{x}_{i}$, where ${x}=\sum_{i=1}^n {x}_i\eps_i$ (its decomposition in the standard basis of $\R^n$).

\medskip

Consider now $\sL=\sL_\rho$ where $\rho$ is the dominant weight $\Lambda_{0}+\cdots+\Lambda_{n-1}=n\Lambda_{0}+\overline{\rho}$ as in \Cref{subsec_AL_gen}.
Denote again, for $w\in W$,  $w=t_{{x}}\overline{w}$ with $\overline{w}$.
We have by Formula \Cref{formula_AL}
\begin{align*}
\label{L(W)}
\sL(w)& =\frac{n^{2}}{2}\Vert{x}\Vert ^{2}+\sL_{\overline{\rho}}(\overline
{w})-n(\mathrm{ht}({x})+\langle\overline{\rho},\overline{w}^{-1}
({x})\rangle)\\
&=
\frac{n^{2}}{2}\Vert{x}\Vert^{2}+\sL_{\overline{\rho}}(\overline{w})-n\langle
\overline{\rho},{x}-\overline{w}^{-1}({x})\rangle.\nonumber
\end{align*}

Also for any $\overline{w}\in{\mathring{W}}$, since the $i$-th coordinate of
$\overline{\rho}$ is equal to $\frac{n+1}{2}-i$, we have
\[
\left\Vert \overline{\rho}\right\Vert ^{2}=\frac{n(n+1)(n-1)}{12}.
\]
In particular, for any $\overline{w}\in{\mathring{W}}$, the $i$-th coordinate of
$\overline{w}(\overline{\rho})$ is equal to $\frac{n+1}{2}-\overline{w}
^{-1}(i)$.\footnote{It is not equal to $\frac{n+1}{2}-\overline{w}(i)$ because
the action of $W_0$ on $\mathbb{Z}^{n}$ is by permutation of the
coordinates.} 
Using $\sum_{i=1}^{n}\overline{w}(i)=\frac
{n(n+1)}{2}$ and 
$\langle\overline{w}(\overline{\rho}),\overline{\rho}\rangle=\langle\overline{w^{-1}}(\overline{\rho}),\overline{\rho}\rangle$,
we obtain
\begin{align*}
\sL(\overline{w})&=\left\Vert \overline{\rho}\right\Vert ^{2}-\langle\overline
{w}(\overline{\rho}),\overline{\rho}\rangle\\
& =\frac{n(n+1)(n-1)}{12}-\sum
_{i=1}^{n}\left(  \frac{n+1}{2}-\overline{w}(i)\right)  \left(  \frac{n+1}
{2}-i\right)  \\
& = \frac{n(n+1)(n-1)}{12}-\sum_{i=1}^{n}i\overline{w}(i)-\frac{n(n+1)^{2}}
{4}+\frac{n+1}{2}\sum_{i=1}^{n}\overline{w}(i)+\frac{n+1}{2}\sum_{i=1}^{n}i \\
& =
\frac{n(n+1)(n-1)}{12}-\frac{n(n+1)^{2}}{4}+\frac{n(n+1)^{2}}{2}-\sum
_{i=1}^{n}i\overline{w}(i)\\
& =\frac{n(n+1)(2n+1)}{6}-\sum_{i=1}^{n}i\overline{w}(i).
\end{align*}
By observing that
\[
\frac{1}{2}\sum_{i=1}^{n}\overline{w}(i)^{2}=\frac{1}{2}\sum_{i=1}^{n}
i^{2}=\frac{n(n+1)(2n+1)}{12},
\]
we get
\begin{equation}
\sL(\overline{w})=\frac{1}{2}\sum_{i=1}^{n}\overline{w}(i)^{2}-\sum_{i=1}
^{n}i\overline{w}(i)+\frac{n(n+1)(2n+1)}{12}.\label{L-wbar)}
\end{equation}
In fact, the previous formula extends to the expression of the atomic
length $\sL$ for $w$ in the affine Weyl group $W$ given in the window
notation (\ref{Window}).

\bigskip

We now define the affine analogue of the notion of entropy of a
permutation\footnote{
Note that the entropy of a permutation is usually defined as twice this quantity
in the literature, see \cite{conway2013sphere}.
}
by the exact same formula.

\begin{Def}\label{def_entropy}
The entropy of the affine permutation $w$ in $W$ is 
$$
E(w)=\frac{1}{2}\sum_{i=1}^{n}\left(  w(i)-i\right)  ^{2},
$$
\end{Def}

We shall prove in the following that
$E(W)=\mathbb{N}$, that is, each integer can be regarded as the entropy of an
affine permutation.

\medskip

\begin{Prop}\label{form_AL_window}
For any $w$ in $W$ given in window notation, we have
$$
\sL(w)=\frac{1}{2}\sum_{i=1}^{n}w(i)^{2}-\sum_{i=1}^{n}iw(i)+\frac
{n(n+1)(2n+1)}{12} = E(w).
$$
\end{Prop}

\begin{proof}
For any $w\in W$, set $c=\frac{n(n+1)(2n+1)}{12}$ and
\[
\mathcal{G}(w)=\frac{1}{2}\sum_{i=1}^{n}w(i)^{2}-\sum_{i=1}^{n}iw(i)+c.
\]
By setting $w=\overline{w}t_{{y}}$, we obtain
\[
w(i)=\overline{w}(i)+n{y}_{i},\quad i=1,\ldots,n
\]
and therefore
\begin{align*}
\mathcal{G}(w) & =\frac{1}{2}\sum_{i=1}^{n}\left(  \overline{w}(i)+n{y}
_{i}\right)  ^{2}-\sum_{i=1}^{n}i\left(  \overline{w}(i)+n{y}_{i}\right)
+c\\
& = \frac{n^{2}}{2}\left\Vert {y}\right\Vert ^{2}+\frac{1}{2}\sum_{i=1}
^{n}\overline{w}(i)^{2}-\sum_{i=1}^{n}i\overline{w}(i)+n\sum_{i=1}
^{n}\overline{w}(i){y}_{i}-n\sum_{i=1}^{n}i{y}_{i}+c.
\end{align*}

By using (\ref{L-wbar)}), this gives
\begin{align*}
\mathcal{G}(w) & = \frac{n^{2}}{2}\left\Vert {y}\right\Vert ^{2}+\sL(\overline
                    {w})-n\sum_{i=1}^{n}{y}_{i}(i-\overline{w}(i)) \\
               & = \frac{n^{2}}{2}\left\Vert
{y}\right\Vert ^{2}+\sL(\overline{w})-n\sum_{i=1}^{n}{y}_{i}\left(
\left(  \frac{n+1}{2}-\overline{w}(i)\right)  -\left(  \frac{n+1}{2}-i\right)
\right)  \\
& =
\frac{n^{2}}{2}\left\Vert {y}\right\Vert ^{2}+\sL(\overline{w})-n\langle
{y},\overline{w}^{-1}(\overline{\rho})-\overline{\rho}\rangle \\
&=\frac{n^{2}
}{2}\left\Vert {y}\right\Vert ^{2}+\sL(\overline{w})-n\langle{y}
,\overline{w}^{-1}(\overline{\rho})\rangle+n\langle{y},\overline{\rho
}\rangle\\
& =
\frac{n^{2}}{2}\left\Vert {y}\right\Vert ^{2}+\sL(\overline{w})-n\langle
\overline{w}({y}),\overline{\rho}\rangle+n\langle{y},\overline{\rho
}\rangle \\
&=\frac{n^{2}}{2}\left\Vert {y}\right\Vert ^{2}+\sL(\overline
{w})-n\langle\overline{w}({y})-{y},\overline{\rho}\rangle.
\end{align*}
Now observe that we have $w=\overline{w}t_{{y}}=w=\overline{w}t_{{y}
}\overline{w}^{-1}\overline{w}=t_{\overline{w}({y})}\overline{w}$. With the
decomposition $w=t_{{x}}\overline{w}$ used in the expression of the atomic
length $\sL$ obtained in (\cite{{CG2022}}), we thus have ${x}
=\overline{w}({y})$ and this gives the desired equality:
\[
\mathcal{G}(w)=\frac{n^{2}}{2}\Vert{x}\Vert ^{2}+\sL(\overline
{w})-n\langle{x}-\overline{w}^{-1}({x}),\overline{\rho}\rangle
=\sL(w).
\]
\end{proof}

\bigskip

\subsection{\texorpdfstring{Universality problem for $\sL$, Euclidean norm and the the 290 theorem}{Universality problem for sL, Euclidean norm and the the 290 theorem}} \label{subsec_eucli_290}

Proposition \ref{form_AL_window} allows us to express the affine atomic length 
associated to the sum of the fundamental weights as a quadratic polynomial in
$n$ variables. Nevertheless, this polynomial is non homogeneous. Deciding
whether a general integral quadratic polynomial in $n$ variables is universal
(i.e. represents all the integers when each of its indeterminates runs over
$\mathbb{Z}$) is a difficult problem less well understood than the case of
quadratic forms (i.e. homegeneous quadratic polynomials).\ We refer the reader
to \cite{ChanOh} for a survey on this question. We are going to see that up to a
translation by a vector in $\mathbb{Z}^{n}$, we can study the universality of
$\sL$ from that of a quadratic form on a subset of $\mathbb{Z}^{n}$.

\medskip

\subsubsection{The sets involved}
Denote by $\mathcal{Q}_{n}$ the root lattice of type $A_{n-1}$, that is
$$
\mathcal{Q}_{n}=\left\{  (x_{1},\ldots,x_{n})\in\mathbb{Z}^{n}\mid
x_{1}+\cdots+x_{n}=0\right\}.
$$

We introduce now three sets that we shall use later.

\medskip

The first set is
\begin{equation}
    D_{n} := \left\{  (y_{1},\ldots,y_{n})\in\mathbb{Z}^{n}\mid\left\{
\begin{array}
[c]{c}
y_{i}\neq y_{j}\mod n,1\leq i<j\leq n\\
y_{1}+\cdots+y_{n}=\frac{n(n+1)}{2}
\end{array}
\right.  \right\}.
\end{equation}
This set is nothing but all the possible window notation vectors of the elements of $W$,  hence $W \simeq D_n$. For $w \in W$ we write 
\begin{equation}\label{vector window notation}
    y_w := (w(1),\dots, w(n)) \in D_n
\end{equation}
 its corresponding window notation vector.

\medskip

The second set is
\begin{equation}
\Delta_{n} := \left\{  (x_{1},\ldots,x_{n})\in\mathbb{Z}^{n}\mid\left\{
\begin{array}
[c]{c}
x_{i}+i\neq x_{j}+j\mod n,1\leq i<j\leq n\\
x_{1}+\cdots+x_{n}=0
\end{array}
\right.  \right\} \subset \mathcal{Q}_n.
\end{equation}

We define now for $i = 1,\dots, n-1$ the polynomial
$$
T_i := x_1 + \dots + x_{i-1} + 2x_i + x_{i+1} + \dots + x_{n-1}.
$$
Then the third set is given by
\begin{equation}
X_{n}:=\left\{  (x_1,\ldots,x_{n-1})\in\mathbb{Z}^{n-1}~|~\left\{
\begin{array}
[c]{c}
x_{i}+i\neq x_{j}+j\mod  n \quad \mathrm{for} \quad 1\leq i<j\leq n-1\\
T_i \neq (n-i) ~ \mathrm{mod} ~ n \quad \mathrm{for} \quad 1\leq i< n-1
\end{array}
\right.  \right\}.
\end{equation}

\bigskip

\subsubsection{The maps and quadratic forms involved}
From Proposition \ref{form_AL_window} it should be clear that we want to know whether the
quadratic polynomial
\begin{equation}\label{def P}
P(y_{1},\ldots,y_{n}) := \frac{1}{2}\sum_{i=1}^{n}y_{i}^{2}-\sum_{i=1}^{n}
iy_{i}+\frac{n(n+1)(2n+1)}{12}
\end{equation}
is universal on $D_n$.

\medskip

We introduce now the two following maps 
$$
\begin{array}{ccccc}
     \mathscr{C} & : & D_n                 & \longrightarrow & \Delta_n  \\
       &   & (y_1,~y_2,~\dots,~y_n) & \longmapsto     & (y_1-1,~ y_2-2,~\dots,~ y_n-n)
\end{array}
$$
and
$$
\begin{array}{ccccc}
     \mathrm{pr} & : & \Delta_n                 & \longrightarrow & X_n  \\
       &   & (x_1,~\dots,~x_n) & \longmapsto     & (x_1,~\dots, ~x_{n-1}),
\end{array}
$$

\medskip

and the two quadratic forms
\begin{equation}\label{form Q}
Q(x_{1},\ldots,x_{n}) := \frac{1}{2}\sum_{i=1}^{n}x_{i}^{2}
\end{equation}
and
\begin{equation}\label{def_q}
q(x_{1},\ldots,x_{n-1}) := \sum_{i=1}^{n-1}x_{i}^{2}+\sum_{1\leq i<j\leq
n-1}x_{i}x_{j}.
\end{equation}

\bigskip

\subsubsection{The statements, examples and problem of Section \ref{subsec_eucli_290}}

The main proposition of this section is the following, which shows at once all the sets and maps involved in the universality of the $\rho$-atomic length in type $A_{n-1}^{(1)}$, where we recall that $W$ is the corresponding affine Weyl group. Recall also that $y_w$ is defined in (\ref{vector window notation}).

\bigskip

\begin{Prop}\label{prop recap}
\label{Prop_reform} We have
\begin{enumerate}
    \item The maps $\mathscr{C}$ and and $\mathrm{pr}$ are bijective. 
    \item $\sL(w)= P(y_w) = Q\big(\mathscr{C}(y_w)\big) = q\big(\mathrm{pr}(\mathscr{C}(y_w))\big)$ for any $w\in W$.
    \item The following are equivalent:
    \begin{enumerate}
        \item $P$ is universal on $D_{n}$,
        \item $Q$ is universal on $\Delta_n$,
        \item $q$ is universal on $X_n$.
\end{enumerate}
\end{enumerate}

\bigskip

To summarize, the following diagram is commutative and the universality of the atomic length $\sL$ on $W$ is equivalent to the universality of each vertical map below

\begin{center}
\begin{tikzcd}
W \arrow[r, "\simeq"] \arrow[rd, "\sL"'] 
  & D_n \arrow[r, "\mathscr{C}"] \arrow[d, "P"', pos=0.35] 
  & \Delta_n \arrow[r, "\mathrm{pr}"] \arrow[ld, "Q"', pos=0.45] 
  & X_n \arrow[lld, "q", pos=0.45] \\
  & \mathbb{Z} & &
\end{tikzcd}
\end{center}

\end{Prop}

\begin{proof}\
\begin{enumerate}
    \item 
    \begin{itemize}
        \item First of all, the map $\mathscr{C}$ is well-defined because 
    $$
    \sum_{i=1}^ny_i - i = \sum_{i=1}^ny_i - \sum_{i=1}^n i = \frac{n(n+1)}{2} - \frac{n(n+1)}{2} = 0,
    $$
    and if $i \neq j$ then $y_i \neq y_j ~\mathrm{mod}~n$, which gives
    $$
    (y_i - i) + i \neq (y_j - j) + j ~\mathrm{mod}~n.
    $$
    With respect to the bijectivity:
    $$
    \mathscr{C}(y) = \mathscr{C}(z) \Longleftrightarrow (y_1-1,\dots,y_n - n) = (z_1-1,\dots,z_n - n) \Longleftrightarrow y = z,
    $$ 
    hence the injectivity. If $x=(x_1,\dots, x_n) \in \Delta_n$ then we claim that $(x_1+1,\dots,x_n+n) \in D_n$ and since clearly $\mathscr{C}(x_1+1,\dots,x_n+n) = x$, we are done. Let us show the claim: 
    $$
    \sum_{i=1}^n x_i + i = \sum_{i=1}^nx_i + \sum_{i=1}^n i = 0 + \frac{n(n+1)}{2} = \frac{n(n+1)}{2},
    $$

    and since $x \in \Delta_n$ one has $x_{i}+i\neq x_{j}+j \mod n$ for $1\leq i<j\leq n$, which is precisely the second condition required for $(x_1+1,\dots,x_n+n)$   to belong to $D_n$.

    \item The map $\mathrm{pr}$ is clearly a bijection between the hyperplane $x_1 + \dots + x_n = 0$ and $\mathbb{Z}^{n-1}$. Therefore, we only need to check that the the conditions $x_{i}+i\neq x_{n}+n\mod n$ for $1\leq i < n$ are the same as $T_i \neq (n-i) ~ \mathrm{mod} ~ n$. This is a direct computation:
    \begin{equation*}
        x_{i}+i\neq x_{n}+n~\mathrm{mod}~n  \Longleftrightarrow ~ x_{i}+i\neq -(x_{1}+\dots +x_{n-1}) + n~\mathrm{mod}~n 
                                              \Longleftrightarrow ~ T_i \neq  ~(n -i)~\mathrm{mod}~n
    \end{equation*}
    \end{itemize}
    \item The first equality follows from Proposition \ref{form_AL_window} and (\ref{def P}). Set $y_{i}=x_{i}+i$. For the second equality it follows from the computation
\begin{align*}
P(y_{1},\ldots,y_{n})&=\frac{1}{2}\sum_{i=1}^{n}(x_{i}+i)^{2}-\sum_{i=1}
^{n}i(x_{i}+i)+\frac{n(n+1)(2n+1)}{12}\\
&=\frac{1}{2}\sum_{i=1}^{n}x_{i}^{2}+\sum_{i=1}^{n}ix_{i}+\frac{1}{2}\sum
_{i=1}^{n}i^{2}-\sum_{i=1}^{n}ix_{i}-\sum_{i=1}^{n}i^{2}+\frac{1}{2}\sum
_{i=1}^{n}i^{2} \\
& =\frac{1}{2}\sum_{i=1}^{n}x_{i}^{2},
\end{align*}
which we apply to our setting, that is $y = y_w$ and $x = f(y_w)$, where the condition $y_i = x_i + i$ is satisfied. 

Let us show now the last equality. In each element of $\Delta_{n}$ we have $x_{n}=-(x_{1}
+\cdots+x_{n})$ which gives
\begin{align*}
Q(x_{1},\ldots,x_{n})& =\frac{1}{2}\sum_{i=1}^{n}x_{i}^{2}=\frac{1}{2}\sum_{i=1}^{n-1}x_{i}^{2}+\frac{1}{2}(x_{1}+\cdots
+x_{n})^{2}\\
& =\sum_{i=1}^{n-1}x_{i}^{2}+\sum_{1\leq i<j\leq n-1}x_{i}x_{j} \\
& = q(x_1,\dots,x_{n-1}).
\end{align*}
    \item This is a direct consequence of the second point.
\end{enumerate}
 
\end{proof}

\begin{Rem}\label{rem_univ}
We show in Proposition \ref{prop recap} that $Q$ is universal on $\Delta_{n}$ if and only if it is universal on $X_n$. 
In fact, the proof also shows that $Q$ is universal on $\mathcal{Q}_n$ if and only $q$ is universal on $\mathbb{Z}^{n-1}$.
\end{Rem}

\medskip

We recall the Fifteen Theorem and the 290-Theorem for positive definite integral quadratic forms.  
The Fifteen Theorem was established by Conway and Schneeberger and first published by Bhargava, while the 290-Theorem was established by Bhargava and Hanke (see \cite{Bhargava2011UniversalQF}).

\begin{Th}
\ 

\begin{itemize}
\item A positive definite integral quadratic form with integer matrix is
universal if and only if it represents the integers of the set%
$$
S_{15}=\{1,2,3,5,6,7,10,14,15\}.
$$

\item A positive definite integral quadratic form is universal if and only if
it represents the integers of the set
$$
S_{290}
=\{1,2,3,5,6,7,10,13,14,15,17,19,21,22,23,26,29,30,31,34,35,37,42,58,93,110,145,203,290\}.
$$
\end{itemize}
\end{Th}

\medskip

\begin{Prop}
\label{Prop_Quniv}
The quadratic form $q$ defined in (\ref{def_q}) is universal on $\mathbb{Z}
^{n-1}$ for any $n\geq5$.
\end{Prop}

\begin{proof}
Observe that the matrix $(a_{i,j})$ of $q$ satisfies $a_{i,j}=\frac{1}{2}$ if
$i\neq j$ and $a_{i,i}=1$. Therefore $q$ has only an half integer
matrix.\ 
Nevertheless, it is easy to check by computer that all the elements of $S_{290}$
are represented by $q$ by considering only $4$ variables $x_1,x_2,x_3,x_4$ 
(and setting the other ones to be zero) 
and letting them run over $\llbracket-8,8\rrbracket$.
\end{proof}

\begin{Prop}\label{non universality 3 variables}
For $ n = 2,3,4$ the form $q$ is not universal on $\mathbb{Z}^{n-1}$, where we recall (\ref{def_q}) that $q(x_1) = x_1^2$, $q(x_1, x_2) = x_1^2 + x_2^2 + x_1x_2$ and
 $q(x_1,x_2,x_3) = x_1^2 + x_2^2 + x_3^2 + x_1x_2 + x_1x_3 + x_2x_3$. 
\end{Prop}

\begin{proof}
    If $n = 2$ then it is obvious. If $n = 3$, then one checks that
$$
x_1^2+x_2^2+x_1x_2 \equiv 0 \text{ or } 1 \mod 3,
$$
for all $x_1,x_2\in \mathbb{Z}$. Therefore $q$ never represents any integer congruent to $2\mod 3$ and is then not universal.

We consider now $n = 4$. By computational experiments on the set $\llbracket -12, 12\rrbracket^3$, we see that all the elements of $S_{290} \setminus \{14,30,110\}$ are represented at least once. But now it is easy to see that 14, 30 and 110 can never be represented by $q$.
\begin{enumerate}
    \item For 14 it is enough to consider all the classes modulo 16. The possible classes are 
    $$
    \{0, 1, 2, 3, 4, 5, 6, 7, 8, 9, 10, 11, 12, 13, 15\}.
    $$
    The class 14 is missing, hence the result.
    \item For 30 it is enough to consider all the classes modulo 32. The missing classes are 
    $$
    M_{32}=\{14,30\}.
    $$
    Hence the result.
    \item For 110 it is enough to consider all the classes modulo 128. The missing classes are 
    $$
    M_{128} = \{14, 30, 46, 56, 62, 78, 94, 110, 120, 126\}.
    $$
    Hence the result. Note that $S_{290} \cap M_{128} = \{14,30, 110\}$, so the others classes are not a problem since they are not to be considered. Note also that modulo 128 we prove immediately the result for the 3 numbers $14,30$ and $110$.
\end{enumerate}

\end{proof}

From \Cref{prop recap}, the universality of the
atomic length $\sL$ can be regarded as a refinement of that of the
quadratic form $Q$ on $\mathcal{Q}_{n}$ (established in \Cref{Prop_Quniv}),
by restricting it to $\Delta_n$.
We therefore ask the question:
\begin{equation}\label{qu_univ}
\text{Is the form $Q$ universal on $\Delta_n$?}
\end{equation}
We will prove in \Cref{Th_EntropyA} that $Q$ is universal on $\Delta_n$ if and only if $n\geq 5$.

\begin{Rem}
Although $Q$ is a symmetric polynomial in $x_{1},\ldots,x_{n}$, the set
$\Delta_{n}$ is no longer invariant under the action of the symmetric
group, making the study of $Q$ on $\Delta_{n}$ much more involved.
\end{Rem}

\begin{Exa}
Let us examinate which intergers are represented by $Q$ on $\Delta_{n}$ for $n=2,3$.

\begin{enumerate}
\item Assume $n=2$, we have to study the integers represented by the quadradic
form $Q(x_{1},x_{2})$ on
\[
\Delta_{2}=\{(x_{1},x_{2})\in\mathbb{Z}^{2}\mid x_{1}+x_{2}=0,x_{1}+1\neq
x_{2}+2\mod 2\}.
\]
since $x_{2}=-x_{1}$ in the elements of $\Delta_{2}$, we always have
$x_{1}+1\neq x_{2}+2\mod 2$ and we are reduced to the integers
represented by the form
\[
q(x_{1})=Q(x_{1},-x_{2})=x_{1}^{2}.
\]
This means, we only obtain the squares.

\item Assume $n=3$, then
\[
\Delta_{3}=\{(x_{1},x_{2}\in\mathbb{Z}^{2}\mid x_{1}+x_{2}+x_{3}
=0,\{x_{1}+1,x_{2}+2,x_{3}+3\}\mod 3=\mathbb{Z}/3\mathbb{Z}\}.
\]
By elementary considerations, one shows that we get the integers of
the form
\[
q(x_{1},x_{2})=Q(x_{1},x_{2},-x_{1}-x_{2})=x_{1}^{2}+x_{2}^{2}+x_{1}x_{2}
\]
where the projection of $(x_{1},x_{2})$ in $(\mathbb{Z}/3\mathbb{Z})^{2}$
belongs to the set $\{(\overline{0},\overline{0}),(\overline{0},\overline
{1}),(\overline{1},\overline{1}),(\overline{1},\overline{2}),(\overline
{2},\overline{0}),(\overline{2},\overline{2})\}$.
\end{enumerate}
\end{Exa}

\bigskip

\section{An additive combinatorics problem and its consequences}
\label{sec_univA}

The goal of this section is to connect the universality problem posed in \Cref{qu_univ}
with an elementary question in additive combinatorics.
We will use a result by Hall to conclude.

\subsection{Background on Hall's theorem}

Let $\Gamma$ be a finite abelian group of order $n$. We use additive notation; in particular, the identity element of $\Gamma$ is denoted by $0$ and the inverse of $g$ in $\Gamma$ is denoted by $-g$. We denote the elements of $\Gamma$ by
$$
\Gamma = \{a_1, a_2,\dots, a_n\}.
$$

Given a finite set $A$ of $\Gamma$, recall that the sumset set $A+A$ is defined by
$$
A+A=\{a+a^{\prime}\mid(a,a^{\prime})\in \Gamma^{2}\}
$$
and the difference set $A-A$ by
$$
A-A=\{a-a^{\prime}\mid(a,a^{\prime})\in \Gamma^{2}\}.
$$

\medskip

In 1952, Hall proved the following theorem, which we shall use later (proof of Theorem \ref{Th_EntropyA}).

\begin{Th}[\cite{hall1952combinatorial}]\label{main th}
	Let $d_1, d_2,\dots, d_n \in \Gamma$ (not necessarily all distinct). The following two statements are equivalent:
	\begin{enumerate}
		\item There exists a permutation of $\Gamma$
		$$
		\sigma = 
		\begin{pmatrix}
		a_1 & a_2 & \dots & a_{n} \\
		b_1 & b_2 & \dots & b_{n}
		\end{pmatrix}
		$$
		such that ~
		$
		b_i - a_i = d_i ~ \text{for all} ~ i =1,2,\dots, n. 
		$
		The vector $\mathsf{d} = (d_1,d_2,\dots,d_n)$ is called the $D$-vector of $\sigma$.
		\item $\sum_{i=1}^n d_i = 0$.
	\end{enumerate}
\end{Th}

\medskip

\begin{Exa}
	Take $\Gamma = \mathbb{Z}/4\mathbb{Z}$ and $\sigma : \Gamma \rightarrow  \Gamma$ defined by
	$$
	\sigma = 
	\begin{pmatrix}
	\overline{0} & \overline{1} & \overline{2} & \overline{3} \\
	\overline{3}  & \overline{1}  &  \overline{0} &  \overline{2}
	\end{pmatrix}.
	$$
	We have then $d_1 =  \overline{3}$, $d_2 =  \overline{0}$, $d_3 =  \overline{2}$ and $d_4 =  \overline{3}$. Summing up all the $d_i$'s we indeed obtain $\overline{8}= \overline{0}$, illustrating the direction $(1) \implies (2)$ of Theorem \ref{main th}. The reverse direction is much trickier: constructing $\sigma$ in terms of the $d_i$'s is not obvious, and it is actually an interesting question, provided $\mathsf{d} = (d_1,d_2,d_3,d_4) \in \Gamma^4$, to find all the permutations of $\Gamma$ such that their $D$-vector belongs to $\mathfrak{S}_4\cdot \mathsf{d}$.
\end{Exa}

\subsection{Proof of the universality of entropy on affine permutations }

In the rest of this section, we will work with
$$
G=(\mathbb{Z}/n\mathbb{Z})^{n},
$$
whose neutral element is
$$
\mathsf{0} =(\overline{0},\ldots,\overline{0}),
$$
and with its subgroup $H_n$ defined by
$$
H_n=\Bigl\{(h_{1},\ldots,h_{n})\in G \ \Bigm|\  h_{1}+\cdots+h_{n}
=\mathsf{0}\Bigr\}.
$$

We will also consider the subset
$$
\mathcal{A}_n =\Bigl\{(a_{1},\ldots,a_{n})\in G \ \Bigm|\  a_i\neq a_j,\ \forall\,1\le i<j\le n
\,\Bigr\}\subset G.
$$

One can observe here that $\mathcal{A}_n$ is just the image in $G$ of the permutations
of the vector $(1,2,\ldots,n)\in\mathbb{Z}^{n}$. It is clear that the
map that associates with any permutation $w \in \mathfrak{S}_{n}$ the element
$(\overline{w(1)},\ldots,\overline{w(n)})\in G$ is a bijection from the
symmetric group $\mathfrak{S}_{n}$ to the set $\mathcal{A}_n$. In particular $\left\vert
\mathcal{A}_n\right\vert =n!$. Also, it is easy to see that for any $a\in \mathcal{A}_n$, we have
$$
a_{1}+\cdots+a_{n}=\overline{\frac{n(n+1)}{2}}.
$$

Another interesting property of the set $\mathcal{A}_n$ is its invariance under
sign change: we have $\mathcal{A}_n=-\mathcal{A}_n$. It follows that
\begin{equation}\label{invariance changement de signe}
\mathcal{A}_n+\mathcal{A}_n=\mathcal{A}_n-\mathcal{A}_n.
\end{equation}

We can now completely solve Question \Cref{qu_univ}:
the $\rho$-atomic length in affine type $A_{n-1}^{(1)}$ is universal exactly when $n\geq5$.

\medskip

\begin{Th}
\label{Th_EntropyA}
Fix $n\geq2$ an integer. 
\begin{enumerate}
    \item\label{point1} We have the sumset equality
        \begin{equation}
        \mathcal{A}_n+\mathcal{A}_n=\mathcal{A}_n-\mathcal{A}_n=H_n.\label{Sumset}
        \end{equation}
     \item\label{point2}
     The following equivalent statements hold:
     \begin{enumerate}
     \item \label{P1} The quadratic form $Q$ is universal on $\Delta_{n}$ if and only if $n\geq5$, 
    \item The entropy on affine permutations is universal if and only if $n\geq5$,
    \item The atomic length $\sL$ is universal on $W$ if and only if $n\geq5$.
     \end{enumerate}
\end{enumerate}
\end{Th}

\begin{proof}\
\begin{enumerate}
\item For every $(a_1,\dots, a_n)\in \mathcal{A}_n$, we have
	$$
	a_{1}+\cdots+a_{n}=\overline{\frac{n(n+1)}{2}}.
	$$
	Since 
	$$\overline{\frac{n(n+1)}{2}}+\overline{\frac{n(n+1)}{2}}=\overline{0},$$ 
	we deduce the first inclusion $\mathcal{A}_n+\mathcal{A}_n\subset H_n$.
	For the reverse inclusion we will show that $H_n \subset \mathcal{A}_n - \mathcal{A}_n$, which, by (\ref{invariance changement de signe}), will prove the statement. Let $h = (h_1,\dots, h_n) \in H_n$, which therefore implies $\sum_{i = 1}^n h_i = \overline{0}$. Let $\alpha_1, \dots, \alpha_n$ denote the elements of $\mathbb{Z}/n\mathbb{Z}$. By Hall's theorem (Theorem \ref{main th}) applied to $\mathbb{Z}/n\mathbb{Z}$, there exists a permutation $\sigma : \mathbb{Z}/n\mathbb{Z} \rightarrow \mathbb{Z}/n\mathbb{Z}$ 
	$$
	\sigma = 
	\begin{pmatrix}
	\alpha_1 & \alpha_2 & \dots & \alpha_{n} \\
	\beta_1 &  \beta_2 & \dots &  \beta_{n}
	\end{pmatrix}
	$$
	such that for every $1 \leq i \leq n$ we have $\beta_i - \alpha_i = h_i$. But clearly $\alpha =(\alpha_1, \dots, \alpha_n)$ and $\beta = (\beta_1, \dots, \beta_n)$ are two elements of $\mathcal{A}_n$. Hence
	$	h = \beta - \alpha.	$
\item The fact that the three assertions are equivalent follows from \Cref{prop recap} and \Cref{def_entropy}.
Let us therefore prove only (a).
Assume that $n\geq 5$. We proceed in two steps.
    \begin{itemize}
    \item The quadratic form $Q$ is universal on $\Delta_{n}$ if and only if,
for each integer $k$ there exists $x=(x_{1},\ldots,x_{n})\in\mathcal{Q}_{n}$
such that
\begin{equation}
\overline{x_{i}+i}\neq\overline{x_{i}+j}\text{ for any }1\leq i<j\leq
n\label{C1}
\end{equation}
and
\begin{equation}
k=\frac{1}{2}\sum_{i=1}^{n}x_{i}^{2}=\frac{1}{2}\left\Vert x\right\Vert
^{2}.\label{C2}
\end{equation}
Write for short $a_{0}=(1,2,\ldots,n)\in\mathbb{Z}^{n}$. Then Condition
(\ref{C1}) is equivalent to
\[
\overline{x}+\overline{a}_{0}\in \mathcal{A}_n\text{.}
\]
Recall that the symmetric group ${\mathring{W}}$ acts on $\mathcal{Q}_{n}$ and
on $H_n$ by permuting the coordinates.\ Also for any $x\in\mathcal{Q}_{n}$ and
any $w\in{\mathring{W}}$, we have $\left\Vert w\cdot x\right\Vert ^{2}
=\left\Vert x\right\Vert ^{2}$. Now consider an
integer $k\in\mathbb{N}$. Since $Q$ is universal on $\mathcal{Q}_{n}$ by
Proposition \ref{Prop_Quniv}, there exists $x\in\mathcal{Q}_{n}$ such that
$k=Q(x)$. For such a $x$ in $\mathcal{Q}_{n}$, we have $\overline{x}\in H_n$.
Also $\mathcal{A}_n$ coincides with the orbit of $\overline{a}_{0}$ under the action of
${\mathring{W}}$. By Point \ref{point1} we know that $H_n=\mathcal{A}_n-\mathcal{A}_n$, there thus exist $w_{k}
,w_{k}^{\prime}$ in ${\mathring{W}}$ such that $\overline{x}=w_{k}^{\prime}
\cdot\overline{a}_{0}-w_{k}^{-1}\cdot\overline{a}_{0}$.

\item Since we have $k=Q(x)$, we get $k=Q(u\cdot x)$ for any
$u\in{\mathring{W}}$ and in particular $k=Q(w_{k}\cdot x)$ for the element
$w_{k}$ of step 1.\ Therefore, to obtain the universality of $Q$ on
$\Delta_{n}$, it suffices to show that
\[
\overline{w_{k}\cdot x}+\overline{a}_{0}=w_{k}\cdot\overline{x}+\overline
{a}_{0}\in \mathcal{A}_n.
\]
This means that, thanks to the action of ${\mathring{W}}$ on $\mathcal{Q}_{n}$,
one has to check that one can always replace the vector $x$ by the vector
$w_{k}\cdot x$ in its orbit under the action of ${\mathring{W}}$ so that
Condition (\ref{C1}) becomes satisfied. Observe that $\mathcal{A}_n$ is stable under the
action of ${\mathring{W}}$ on $H_n$.\ Therefore, we get the equivalence
\[
w_{k}\cdot\overline{x}+\overline{a}_{0}\in \mathcal{A}_n\Longleftrightarrow\overline
{x}+w_{k}^{-1}\cdot\overline{a}_{0}\in \mathcal{A}_n.
\]
Also $\mathcal{A}_n$ coincides with the orbit of $\overline{a}_{0}$ under the action of
${\mathring{W}}$ so that $\overline{x}+w_{k}^{-1}\cdot\overline{a}_{0}\in \mathcal{A}_n$ if
any only if there exists $w^{\prime}\in{\mathring{W}}$ such that $\overline
{x}=w^{\prime}\cdot\overline{a}_{0}-w_{k}^{-1}\cdot\overline{a}_{0}$ which is
guaranteed by our hypothesis $H_n=\mathcal{A}_n-\mathcal{A}_n$ by choosing $w^{\prime}=w_{k}$ as explained
in \ref{point1}. In conclusion for any $n\geq5$, the sumset equality $H=\mathcal{A}_n-\mathcal{A}_n$ implies that our quadratic
form $Q$ is universal on $\Delta_{n}$.
\end{itemize} 
Conversely, assume $n\leq 4$.
By Proposition \ref{Prop_Quniv} and Proposition \ref{non universality 3 variables}, 
some integers are not represented by $q$ on $\mathbb{Z}^{n-1}$, so on $X_n$ neither (since $X_n \subset \mathbb{Z}^{n-1}$). 
Therefore, by Proposition \ref{prop recap}, $Q$ is not universal on $\Delta_n$.
\end{enumerate}

\end{proof}

\section{\texorpdfstring{A polynomial expression of the atomic length in affine type $A$}{A polynomial expression of the atomic length in affine type A}}
\label{sec_pol_formulas}

In this section, we use a result by the third author \cite{Jacon:size} to establish a
polynomial expression for the atomic length in type $A_{n-1}^{(1)}$ associated
to any dominant weight $\La$. 
Here again, this computes
the size of relevant generalisations of $n$-cores.
In fact, it is a direct higher level analogue of the Garvan-Kim-Stanton formula \Cref{GKS_pol} giving the size of
an $n$-core. 
We conjecture that in this general setting, the atomic length is also universal 
when a simple condition on the weight $\La$ is satisfied. 

\subsection{Symbols and abaci}

A {\it partition} $\lambda=(\lambda_1,\ldots,\lambda_r)$ of size $k\in \mathbb{N}$ is a sequence of non increasing integers of total sum $|\lambda|:=k$.  
If $\ell\in \mathbb{N}$, an {\it $\ell$-partition} (or multipartition) $\boldsymbol{\lambda}$ is a sequence of  partitions $(\lambda^1,\ldots,\lambda^{\ell})$ 
such that 
$|\boldsymbol{\lambda}|:=\sum_{1\leq i\leq \ell} |\lambda^i|=k$. 
The integer  $k$ is also called the size of the $\ell$-partition $\boldsymbol{\lambda}$. 
 The set of $\ell$-partitions of size $k$ is denoted by $\Pi^\ell (k)$ and the set of $\ell$-partitions by $\Pi^\ell$.

   \medskip
   
  By definition,  a  {\it symbol}  of {\it charge}  $s\in \mathbb{Z}$  is an infinite sequence of integers 
 $X=(\beta_i)_{i<s}$ such that:
 \begin{enumerate}
 \item For all $i< s$, we have $\beta_{i-1}<\beta_i$ (that is, $X$ is a strictly increasing sequence),
 \item There exists $N< s$ such that for all $j\leq N$, we have $\beta_j=j$.
 \end{enumerate}
The symbol  $X=(\beta_i)_{i<s}$ such that $\beta_i=i$ for all $i<s$ is called the {\it trivial symbol}.
A symbol may be conveniently represented using its abacus configuration.  We associate to a symbol $X$
a horizontal runner of black and white beads indexed by $\mathbb{Z}$,
where a bead indexed by $a\in \mathbb{Z}$  
is black if and only if $a\in X_j$ (we will say that the {\it position} of the bead is $a$). 
\begin{Exa}
The abacus associated to the symbol $X=(\ldots,-3,-2,-1,2, 3,5,7)$  of charge $4$ is
        \begin{center}
\begin{tikzpicture}[
  scale=0.5,
  bb/.style={draw,circle,fill,minimum size=2.5mm,inner sep=0pt,outer sep=0pt},
  wb/.style={draw,circle,fill=white,minimum size=2.5mm,inner sep=0pt,outer sep=0pt}
]

\foreach \x in {-3,...,18}
  \node at (\x,-1) {\x};

\foreach \x in {-3,...,18}
  \node[wb] at (\x,0) {};

\foreach \x in {-3,-2,-1,2,3,5,7}
  \node[bb] at (\x,0) {};

\node at (-4,0) {$\ldots$};
\node at (19,0) {$\ldots$};

\end{tikzpicture}
\end{center}

\end{Exa}
The charge of a symbol $X$
can be conveniently read off
from its associated abacus as follows. 
For each black bead, replace recursively the leftmost white bead to its left (if it exists) with a black bead, and replace the black bead itself with a white bead. 
This way, we obtain the abacus of a trivial symbol. The charge of 
$X$ is then defined as the charge of this trivial symbol, that is, it is the index of the leftmost white bead.

\medskip

This generalises to $\ell$-partitions.
An {\it $\ell$-symbol} is a collection of $\ell$  symbols
 $${\bf X}=(X^1,\ldots,X^{\ell}).$$
The {\it multicharge} (or $\ell$-charge) of the symbol is the $\ell$-tuple $(s_1,\ldots,s_{\ell})\in \mathbb{Z}^\ell$ where for all $j=1,\ldots,\ell$, the number $s_j$ is the charge of $X_j =(\beta^j_i)_{i< s_j}$. 
 An {\it $\ell$-symbol} ${\bf X}=(X^1,\ldots,X^{\ell})$ can be  represented using its abacus configuration. 
 In this way, we associate to each $X^j$, for $j=1,\ldots,\ell$, an abacus as above and we write them from bottom to top, aligned so that the beads in the same column are indexed by the same integer.   
 We call the associated object an {\it $\ell$-abacus}.
  \begin{Exa}\label{firstabacus}
  Let $\ell=3$ and let us consider the following $3$-symbol:
  $${\bf X}=((\ldots,-1,0,2,4,6),(\ldots,-1,0,3,4),(\ldots,-1,0,2,5)).$$
  The associated $3$-abacus is:
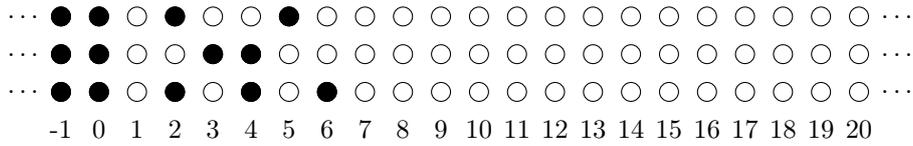
\begin{center}
\begin{tikzpicture}[
  scale=0.5,
  bb/.style={draw,circle,fill,minimum size=2.5mm,inner sep=0pt,outer sep=0pt},
  wb/.style={draw,circle,fill=white,minimum size=2.5mm,inner sep=0pt,outer sep=0pt}
]

\foreach \x in {-1,...,20}
  \node at (\x,-1) {\x};

\foreach \x in {-1,...,20} \node[wb] at (\x,0) {};
\foreach \x in {6,4,2,0,-1} \node[bb] at (\x,0) {};
\node at (-2,0) {$\ldots$};
\node at (21,0) {$\ldots$};

\foreach \x in {-1,...,20} \node[wb] at (\x,1) {};
\foreach \x in {4,3,0,-1} \node[bb] at (\x,1) {};
\node at (-2,1) {$\ldots$};
\node at (21,1) {$\ldots$};

\foreach \x in {-1,...,20} \node[wb] at (\x,2) {};
\foreach \x in {5,2,0,-1} \node[bb] at (\x,2) {};
\node at (-2,2) {$\ldots$};
\node at (21,2) {$\ldots$};

\end{tikzpicture}
\end{center}
\end{Exa}

To each  symbol $X =(\beta_i)_{i<s}$  (and thus to each abacus)   of charge $s$ we can canonically associate 
a partition $\lambda (X)=(\lambda_1,\ldots,\lambda_r)$ such that for all $i\geq 1$, we have $\lambda_i=\beta_{s-i}+i-s$. Note that if  $k$ is large enough then  $\lambda_k=0$. 
Regarding the abacus associated to the set of $\beta$-numbers, the parts of the partition are easily obtained by counting the numbers of white beads at the left of each black bead. 
Conversely, to any partition $(\lambda_1,\ldots,\lambda_r)$, we can associate a set of $\beta$-numbers  (and thus an abacus). Let  $s\in \mathbb{Z}$.  Then we define:  
$$X^s (\lambda) =(\beta_i)_{i< s}$$
where  for all $i=1,\ldots,m$, we have $\beta_{m-i}=\lambda_i-i+m$.  

\begin{Rem}
    On can easily check whether a charged partition $\lambda$ is a $n$-core by looking at its associated abacus. We indeed have the equivalence: $\lambda$ is a $n$-core if and only if for each black bead at position $k$ in its abacus, we have another black bead at position $k-n$.
\end{Rem}    

By extension, to each $\ell$-symbol ${\bf X}=(X^1,\ldots,X^{\ell})$, we can associate an $\ell$-partition.
Conversely, to each  multipartition  $\boldsymbol{\lambda}^\ell=(\lambda^1,\ldots,\lambda^{\ell})$ and multicharge ${\bf s}^\ell=(s_1,\ldots,s_{\ell})$, one can attach an $\ell$-symbol 
$${\bf X}^{{\bf s}^\ell} (\boldsymbol{\lambda}^l)=(X^{s_1} (\lambda^1),\ldots,X^{s_{\ell}} (\lambda^{\ell})).$$ 
 \begin{Exa}
In Example \ref{firstabacus}, we find ${\bf X}={\bf X}^{(4,3,3)} ((3,2,1),(2,2),(3,1))$. 
 \end{Exa}

\subsection{\texorpdfstring{The notion of $(n,\boldsymbol{s})$-core}{The notion of (n,s)-core}}\label{subsec_ns_core}

Let $s\in\Z$, and denote $\mathbb{Z}^{\ell}[s] := \{(s_1,\ldots,s_{\ell}) \in \mathbb{Z}^{\ell}\mid s_1+\cdots+s_{\ell}=s\}$.
Fix ${\bf s}^{\ell}\in\Z^\ell[s]$ and $\boldsymbol{\lambda}^{\ell} \in \Pi^{\ell}$. We associate to this datum an
$n$-partition $\boldsymbol{\lambda}_n$ together with a multicharge
${\bf s}_n\in \mathbb{Z}^n[s]$ as follows.
Start with the abacus associated to
$(\boldsymbol{\lambda}^{\ell},{\bf s}^{\ell})$. Then:
\begin{itemize}
\item 
For each $k\in\Z$, Consider the rectangle $R_k$ containing all beads indexed by $x$ such that $x-(x \mod e) = k$.
\item Rotate each rectangle $R_k$ $90$ degrees anticlockwise.
\item We get an $n$-abacus, which is the $n$-abacus of
$(\bmu,{\bf t})$ where ${\bmu}=(\mu^1,\ldots,\mu^n)$ and
${\bf t}=(t^1,\ldots,t^n)$.
\item We then define $\boldsymbol{\lambda}_n:=(\mu^n,\ldots,\mu^1)$ and
${\bf s}_n:=(t^n,\ldots,t^1)$.
\end{itemize}

\begin{Exa}\label{exz}
Assume that $\ell=2$ and $n=3$, $\boldsymbol{\lambda}^{\ell}=((3,1),(2,1))$ and
${\bf s}^{\ell}=(0,0)\in \mathbb{Z}^2[0]$. The $\ell$-abacus associated with
$(\boldsymbol{\lambda}^{\ell},{\bf s}^{\ell})$:
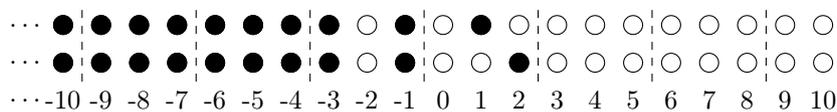
  \begin{center}
    \begin{tikzpicture}[
  scale=0.5,
  bb/.style={draw,circle,fill,minimum size=2.5mm,inner sep=0pt,outer sep=0pt},
  wb/.style={draw,circle,fill=white,minimum size=2.5mm,inner sep=0pt,outer sep=0pt}
]

\foreach \x/\lab in {-9/-10,-8/-9,-7/-8,-6/-7,-5/-6,-4/-5,-3/-4,-2/-3,-1/-2,0/-1,1/0,2/1,3/2,4/3,5/4,6/5,7/6,8/7,9/8,10/9,11/10}
  \node at (\x,-1) {\lab};
\node at (-10,-1) {$\ldots$};

\foreach \x in {-9,...,11} \node[wb] at (\x,0) {};
\foreach \x in {3,0,-2,-3,-4,-5,-6,-7,-8,-9} \node[bb] at (\x,0) {};
\node at (-10,0) {$\ldots$};

\foreach \x in {-9,...,11} \node[wb] at (\x,1) {};
\foreach \x in {2,0,-2,-3,-4,-5,-6,-7,-8,-9} \node[bb] at (\x,1) {};
\node at (-10,1) {$\ldots$};

\foreach \x in {0.5,3.5,-2.5,-5.5,6.5,9.5,-8.5}
  \draw[dashed] (\x,-0.5) -- (\x,1.5);

\end{tikzpicture}
\end{center}
and then after rotation, the $n$-abacus associated to $(\boldsymbol{\lambda}_n,{\bf s}_n)$
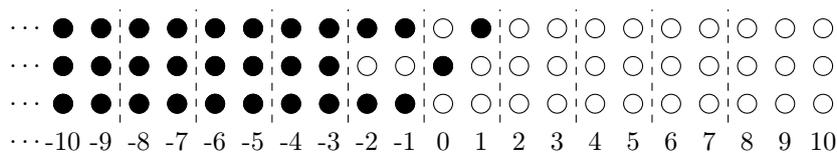
 \begin{center}
    \begin{tikzpicture}[
  scale=0.5,
  bb/.style={draw,circle,fill,minimum size=2.5mm,inner sep=0pt,outer sep=0pt},
  wb/.style={draw,circle,fill=white,minimum size=2.5mm,inner sep=0pt,outer sep=0pt}
]

\foreach \x/\lab in {-9/-10,-8/-9,-7/-8,-6/-7,-5/-6,-4/-5,-3/-4,-2/-3,-1/-2,0/-1,1/0,2/1,3/2,4/3,5/4,6/5,7/6,8/7,9/8,10/9,11/10}
  \node at (\x,-1) {\lab};
\node at (-10,-1) {$\ldots$};

\foreach \x in {-9,...,11} \node[wb] at (\x,0) {};
\foreach \x in {0,-1,-2,-3,-4,-5,-6,-7,-8,-9} \node[bb] at (\x,0) {};
\node at (-10,0) {$\ldots$};

\foreach \x in {-9,...,11} \node[wb] at (\x,1) {};
\foreach \x in {1,-2,-3,-4,-5,-6,-7,-8,-9} \node[bb] at (\x,1) {};
\node at (-10,1) {$\ldots$};

\foreach \x in {-9,...,11} \node[wb] at (\x,2) {};
\foreach \x in {2,0,-1,-2,-3,-4,-5,-6,-7,-8,-9} \node[bb] at (\x,2) {};
\node at (-10,2) {$\ldots$};

\foreach \x in {-7.5,-5.5,-3.5,-1.5,0.5,2.5,4.5,6.5,8.5}
  \draw[dashed] (\x,-0.5) -- (\x,2.5);

\end{tikzpicture}
\end{center}
So we have $\boldsymbol{\lambda}_n=((1),(2),\emptyset)$ and ${\bf s}_n=(1,-1,0)$. 

\end{Exa}

  It is clear that here we have a bijection 
  $$ \begin{array}{cccc}
 \varphi:  &  \Pi^\ell \times \mathbb{Z}^\ell [s] & \to &   \Pi^n \times \mathbb{Z}^n [s]  \\
 & (\boldsymbol{\lambda}^\ell,{\bf s}^\ell) & \mapsto &  (\boldsymbol{\lambda}_n,{\bf s}_n) 
\end{array}  $$

\begin{itemize}
\item  The $n$-partition $\boldsymbol{\lambda}_n$ is the {\it $n$-quotient} of $(\boldsymbol{\lambda}^\ell,{\bf s}^\ell)$. 
\item   The pair $(\bnu^\ell,{{\bf s}'}^\ell)\in \Pi^\ell \times \mathbb{Z}^\ell [s]$ such that $\varphi (\bnu^\ell,{{\bf s}'}^\ell)=(\boldsymbol{\emptyset},{\bf s}_n)$
 is called   the {\it $n$-core} of $(\boldsymbol{\lambda},{\bf s}^\ell)$.    It is uniquely determined by ${\bf s}_n$ which is called the {\it $n$-core multicharge} of $(\boldsymbol{\lambda}^\ell,{\bf s}^\ell)$ (in the following, we will often say that the $n$-core $(\boldsymbol{\lambda}^\ell,{\bf s}^\ell)$  is associated to ${\bf s}_n$).
  \end{itemize}

\begin{Def}
A pair $(\bnu^\ell,{{\bf s}}^\ell)\in \Pi^\ell \in \mathbb{Z}^\ell [s] $ is an 
$(n,\boldsymbol{s}^\ell)$-core if  $\varphi (\bnu^\ell,{{\bf s}}^\ell)=(\boldsymbol{\emptyset},{\bf s}_n)$ for an $n$-tuple ${\bf s}_n \in \mathbb{Z}^n [s]$. 
\end{Def}

One can show that if $(\bnu^\ell,{{\bf s}}^\ell)\in \Pi^\ell \in \mathbb{Z}^\ell [s] $ is an $(n,\boldsymbol{s}^\ell)$-core, then we must have 
$$s_1\leq s_2 \leq \ldots \leq s_\ell < s_1+n.$$
 One can easily check whether a  pair $(\bnu^\ell,{{\bf s}}^\ell)\in \Pi^\ell \in \mathbb{Z}^\ell [s] $ is an $(n,\boldsymbol{s}^\ell)$-core looking at its $\ell$-abacus. In this $\ell$-abacus:
 \begin{itemize}
 \item For the $\ell-1$ first runners (starting from the bottom, that is the abaci associated to $X^1,\ldots,X^{l-1}$), for each black bead in position $i$, there must be another black bead at the top of it (that is in the abacus just above)  in the same position 
 \item For the top runner (that is the abacus associated to $X^{\ell}$), for each black bead in position $i$, there must be another black bead in the first (bottom) runner at position $i-e$. 
\end{itemize}

\begin{Rem}
If $\ell=1$, the $n$-quotient defined above agrees with the usual notion of quotient. 
Moreover, the $n$-core of a pair $(\lambda,s)$ 
is always of the form $(\nu,s)$ where $\nu$ is  the usual $n$-core partition of $\lambda$. 
\end{Rem}

\begin{Exa}
Resuming the previous example the $n$-core of $(\boldsymbol{\lambda}^\ell,{\bf s}^\ell)$ have the following $\ell$-abacus:
 
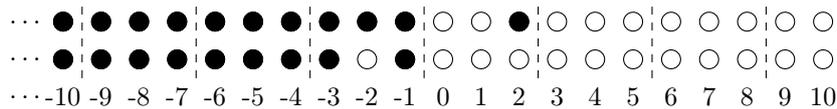
\begin{center}\begin{tikzpicture}[
  scale=0.5,
  bb/.style={draw,circle,fill,minimum size=2.5mm,inner sep=0pt,outer sep=0pt},
  wb/.style={draw,circle,fill=white,minimum size=2.5mm,inner sep=0pt,outer sep=0pt}
]

\foreach \x/\lab in {-9/-10,-8/-9,-7/-8,-6/-7,-5/-6,-4/-5,-3/-4,-2/-3,-1/-2,0/-1,1/0,2/1,3/2,4/3,5/4,6/5,7/6,8/7,9/8,10/9,11/10}
  \node at (\x,-1) {\lab};
\node at (-10,-1) {$\ldots$};

\foreach \x in {-9,...,11} \node[wb] at (\x,0) {};
\foreach \x in {0,-2,-3,-4,-5,-6,-7,-8,-9} \node[bb] at (\x,0) {};
\node at (-10,0) {$\ldots$};

\foreach \x in {-9,...,11} \node[wb] at (\x,1) {};
\foreach \x in {3,0,-1,-2,-3,-4,-5,-6,-7,-8,-9} \node[bb] at (\x,1) {};
\node at (-10,1) {$\ldots$};

\foreach \x in {0.5,3.5,-2.5,-5.5,6.5,9.5,-8.5}
  \draw[dashed] (\x,-0.5) -- (\x,1.5);

\end{tikzpicture}
\end{center}
This $n$-core is $(((1),(2)),(-1,1))$ and the associated $n$-core multicharge is $(0,-1,1)$. 
\end{Exa}

The study of $(n,\bs)$-cores is motivated by the following theorem, 
proved in \cite[Corollary 4.4]{JaconLecouvey2021},
which is an analogue of the Nakayama conjecture for Hecke algebras.
More precisely, let $\cH_{(n,\bs)}$ be the non-semisimple Ariki-Koike algebra
specialised at $(\zeta,\zeta^{s_0},\ldots, \zeta^{s_{\ell-1}})$
where $\zeta$ is a primitive $n$-th root of unity and $\bs=(s_0,\ldots, s_{\ell-1})$.

\begin{Th}\label{Th_blocks}
Two $\ell$-partitions of the same rank are in the same block of $\cH_{(n,\bs)}$
if and only if they have the same $(n,\bs)$-core.
\end{Th}

In particular, the simplest blocks, \textit{defect zero} blocks,
contain exactly one $\ell$-partition which is an $(n,\bs)$-core.

\subsection{\texorpdfstring{Atomic length of a general dominant weight in affine type $A$}{Atomic length of a general dominant weight in affine type A}} \label{Subsec_AtomLengthGeneralDom}

The $(n,\bs)$-cores are higher level generalisations of core partitions 
are the relevant objects for parametrising the orbits of
the dominant weights $\La$ of the affine root systems in type $A$ \cite{JaconLecouvey2021}.\ 
In fact, under this parametrisation,
the size of an $(n,\bs)$-core (i.e., the number of boxes) coincides with the $\La$-atomic length of the corresponding affine Weyl group element. 
In \cite{Jacon:size}, explicit formulas for the size of the $(n,\boldsymbol{s})$-cores are established.

\medskip

Let us give more details about this construction. Consider an integer $\ell$ such
that $1\leq \ell\leq n$ and a $\ell$-tuple $\boldsymbol{s}=(s_{1},\ldots,s_{\ell}
)$\footnote{To simplify the notation, we will drop in the rest of this section
the super and lower scripts for the multicharges in $\mathbb{Z}^{\ell}$ and
$\mathbb{Z}^{n}$. That is, we will simply write $\boldsymbol{s}$ and
$\boldsymbol{t}$ instead of $\boldsymbol{s}^{\ell}\in\mathbb{Z}^{\ell}$ and
$\boldsymbol{t}_{n}\in\mathbb{Z}^{n}$ when there is no risk of confusion.} of
integers such that $0\leq s_{1}\leq s_{1}\leq\cdots\leq s_{\ell}<n$ and
$s_{1}+\cdots+ s_{\ell}=s$.\ This defines in particular a partition $\kappa
=(s_{\ell}\geq\cdots\geq s_{1})$ with $s$ boxes contained in the rectangle
$(n-1)^{\ell}$ and a dominant weight of type $A_{n-1}^{(1)}$
\[
\Lambda(\boldsymbol{s})=\Lambda_{s_{1}}+\cdots+\Lambda_{s_{\ell}}.
\]
Also the transposed partition $\kappa^{\prime}$ has $n$ parts and can be
written $\kappa^{\prime}=(s_{n}^{\prime}\geq\cdots\geq s_{1}^{\prime})$ with
parts in $\{0,\ldots,\ell\}$. Then the orbit of the dominant weight $\Lambda
($$\boldsymbol{s}$$)$ under the action of $W$ is parametrized by the
$n$-tuples $\boldsymbol{t}=(t_{1},\ldots,t_{n})$ in $\mathbb{Z}^{n}$ summing
up $s$ and whose distribution modulo $n$ coincides with $\kappa^{\prime}$ or
equivalently with $\boldsymbol{s}^{\prime}=(s_{1}^{\prime}\leq\cdots\leq
s_{n}^{\prime})$. This is in particular a subset of $\mathbb{Z}^{n}[s]$ as
defined in \Cref{subsec_ns_core}. This means that the residues modulo
$\ell$ of a relevant $n$-tuple $\boldsymbol{t}$ are the same than those of
$\boldsymbol{s}^{\prime}$ up to permutation and $t_{1}+\cdots+t_{n}=s$.\ 
\medskip

In fact the bijection between the elements of the orbit $\mathcal{O}
(\Lambda(\boldsymbol{s}))$ and the $n$-tuples $\boldsymbol{t}$ is easy to
describe. The dominant weight $\Lambda(\boldsymbol{s})$ corresponds to the
$n$-tuple $\boldsymbol{s}^{\prime}$ and for any $w$ in $W$, the $n$-tuple
$\boldsymbol{t}$ associated to $w\cdot\Lambda(\boldsymbol{s})$ is obtained by
computing the action of $w$ on $\boldsymbol{s}^{\prime}$ where $W$ acts on
$\mathbb{Z}^{n}$ by
\begin{gather*}
\mathtt{s}_{i}(t_{1},\ldots,t_{i},t_{i+1},\ldots,t_{n})=(t_{1},\ldots
,t_{i+1},t_{i},\ldots,t_{n})\text{ for }i=1,\ldots n\text{ and}\\
\mathtt{s}_{0}(t_{1},t_{2}\ldots,t_{n-1},t_{n})=(t_{n}-\ell,t_{2}\ldots
\ldots,t_{n-1},t_{1}+\ell).
\end{gather*}
To get a bijection between the elements of $\mathcal{O}(\Lambda(\boldsymbol{s}
))$ and the $(n,\boldsymbol{s})$-cores of \Cref{subsec_ns_core}, it
then suffices to consider for each such $n$-tuple $\boldsymbol{t}$, the
$(n,\boldsymbol{s})$-core $\varphi^{-1}(\boldsymbol{\emptyset},\boldsymbol{t}
)$. Now given an element $w$ in the affine Weyl group $W$, the atomic length
$\sL_{\Lambda(\boldsymbol{s})}(w)$ corresponds to the number of simple roots
appearing in the decomposition of $\Lambda(\boldsymbol{s})-w\left(
\Lambda(\boldsymbol{s})\right)  $ on the basis of simple roots. It was then
established in \cite{JaconLecouvey2021} that this number equates the size  of
$\varphi^{-1}(\boldsymbol{\emptyset},\boldsymbol{t})$ where $\boldsymbol{t}
=w(\boldsymbol{s}^{\prime})$ under the previous action of $W$ on
$\mathbb{Z}^{n}$.

\begin{Def}
Denote by $D_{\boldsymbol{s}}$ the set of $n$-tuples in $\mathbb{Z}^{n}$ whose
distribution modulo $\ell$ is given by $\kappa^{\prime}$.
\end{Def}

\begin{Exa}
Assume $n=5$ and $\ell=3$. Consider $\kappa=(4,2,2)$. Then $\kappa^{\prime
}=(3,3,1,1,0)$ and $D_{\boldsymbol{s}}$ contains exactly the $5$-tuples
summing up $8$ with three residues modulo $3$ equal to $0$, two residues equal
to $1$ and no residue equal to $2$. Observe that in $\kappa^{\prime}$ the
multiplicities of the components equal to $1,\ldots,\ell-1$ are given by the
distribution of the residues but those of $0$ and $\ell$ are determined by
$\left\vert \kappa^{\prime}\right\vert $ and the number of residues equal to
$0$.
\end{Exa}

The following proposition is a reformulation of a result established in
\cite{Jacon:size}\footnote{Observe here that the convention for the action of
$\mathtt{s}_{0}$ in \cite{Jacon:size} is
\[
\mathtt{s}_{0}(t_{1},t_{2}\ldots,t_{n-1}t_{n})=(t_{n}+\ell,t_{2}\ldots
\ldots,t_{n-1},t_{1}-\ell).
\]
Both conventions coincide up to conjugation by the sign flip for the
coordinates in $\mathbb{Z}^{n}$.\ }.

\begin{Prop}
\label{Prop_PolyL}
With the previous notation the atomic length $\sL_{\Lambda(\boldsymbol{s})}$
admits the following polynomial expression in the coordinates of the elements
in $D_{\boldsymbol{s}}$
\[
\sL_{\Lambda(\boldsymbol{s})}(w)=P_{\boldsymbol{s}}(\boldsymbol{t})=\frac
{n}{2\ell}\sum_{i=1}^{n}t_{i}^{2}-\sum_{i=1}^{n}(i-1)t_{i}-c_{\boldsymbol{s}}
\]
where $\boldsymbol{t}=(t_{1},\ldots,t_{n})$ is the element of
$D_{\boldsymbol{s}}$ associated to $w\cdot\Lambda(\boldsymbol{s})$ by the
previous construction and $c_{\boldsymbol{s}}$ is the normalising constant
such that $P_{\boldsymbol{s}}(s_{1}^{\prime},\ldots,s_{n}^{\prime})=0$, i.e.
\[
c_{\boldsymbol{s}}=\frac{n}{2\ell}\sum_{i=1}^{n}(s_{i}^{^{\prime}})^{2}
-\sum_{i=1}^{n}(i-1)s_{i}^{\prime}.
\]
\end{Prop}

\begin{Exa}\
\label{exa_Jac_GKS}
\begin{enumerate}
\item Assume that $\ell=1$ and $\bs=(s)$ with $0\leq s< n$,
so that $\kappa = (s)$ and $\kappa' = (0,\ldots,0,1,\ldots,1)$ ($s$ occurences of $1$).
Then $\Lambda(\boldsymbol{s})=\Lambda
_{s\mod n}$  and 
\begin{align*}
\sL_{\Lambda_{}(\bs)}(w) &
=P_{\boldsymbol{s}}(\boldsymbol{t})=
\frac{n}{2}\sum_{i=1}^{n}t_{i}^{2}-\sum_{i=1}^{n}(i-1)t_{i} + \frac{s}{2}(n-1) - \frac{s^2}{2}
\end{align*}
where $\boldsymbol{t}=(t_{1},\ldots,t_{n})\in\mathbb{Z}^{n}$ is such that
$t_{1}+\cdots+t_{n}=0$ as expected.
Observe that for $s=0$, this is essentially the Garvan-Kim-Stanton formula given in the introduction\footnote{
The only difference is a "plus" sign instead of a "minus" sign, but this is just a consequence of our choice of conventions (see previous footnote).
For instance, in \cite{Jacon:size}, the formula specialises precisely to \Cref{GKS_pol} at $s=0$.
This does not really matter, since universality of one polynomials is equivalent to universality of the other.
}, see \Cref{GKS_pol}.
\item Assume $\ell=n$ and $\boldsymbol{s}=(0,1,\ldots,n-1)$, so that $\La(\bs) =\rho$.
Then
$\boldsymbol{s}^{\prime}=\boldsymbol{s}$ and $\left\vert \boldsymbol{s}
\right\vert =\frac{n(n-1)}{2}$. We obtain
\[
\sL_{\La(\bs)}(w)= \sL(w)=P_{\boldsymbol{s}}(\boldsymbol{t})=\frac
{1}{2}\sum_{i=1}^{n}t_{i}^{2}-\sum_{i=1}^{n}(i-1)t_{i}-\frac{n(n-1)(2n-1)}{12}
\]
since
\[
c_{\boldsymbol{s}}=\frac{1}{2}\sum_{i=1}^{n}(i-1)^{2}-\sum_{i=1}
^{n}(i-1)(i-1)=\frac{n(n-1)(2n-1)}{12}.
\]
in this case. 
The set $D_{\boldsymbol{s}}$ is the subset of $\mathbb{Z}^{n}$ 
whose elements have distinct
coordinates modulo $n$ and summing up to $\frac{n(n-1)}{2}$. This is
slightly different from the expression obtained in Proposition
\ref{form_AL_window}, which was considering the same dominant weight.
This comes from the fact the set $D$ considered in Proposition
\ref{form_AL_window} is the subset of $\mathbb{Z}^{n}$ whose elements have distinct
coordinates modulo $n$ and summing up to $\frac{n(n+1)}{2}$ (that is the
orbit of $(1,2,\ldots,n)$).\ Thus to get the expression of
Proposition \ref{form_AL_window}, we have to set $w_{i}=t_{i}+1$ which gives
\begin{align*}
P_{\boldsymbol{s}}(\boldsymbol{t}) &  =\frac{1}{2}\sum_{i=1}^{n}t_{i}^{2}
-\sum_{i=1}^{n}(i-1)t_{i}+\frac{n(n-1)(2n-1)}{12}\\
&  =\frac{1}{2}\sum_{i=1}^{n}(w_{i}-1)^{2}-\sum_{i=1}^{n}(i-1)(w_{i}
-1)+\frac{n(n-1)(2n-1)}{12}\\
&  =\frac{1}{2}\sum_{i=1}^{n}w_{i}^{2}-\sum_{i=1}^{n}iw_{i}+\frac{n}{2}
+\sum_{i=1}^{n}(i-1)+\frac{n(n-1)(2n-1)}{12}\\
&  =\frac{1}{2}\sum_{i=1}^{n}w_{i}^{2}-\sum_{i=1}^{n}iw_{i}+\frac{n}{2}
+\frac{n(n-1)}{2}+\frac{n(n-1)(2n-1)}{12}\\
&  =\frac{1}{2}\sum_{i=1}^{n}w_{i}^{2}-\sum_{i=1}^{n}iw_{i}+\frac
{n(n+1)(2n+1)}{12}
\end{align*}
as expected. 
\end{enumerate}
\end{Exa}

Combining \Cref{Th_EntropyA}, \Cref{Th_blocks} and Point 2. of \Cref{exa_Jac_GKS}
yields the following immediate Corollary,
which is an analogue of the famous \cite[Corollary 1]{GO1996}.

\begin{Cor}\label{cor_blocks}
For all $n\geq 5$ and $\bs=(0,1,2,\ldots, n-1)$, the Ariki-Koike algebra $\cH_{(n,\bs)}$ has a defect zero block. 
\end{Cor}

\subsection{\texorpdfstring{Dilation of the lattice $\mathbb{Z}^{n}$}{Dilation of the lattice Zn}}

The goal of this paragraph is to see that the previous atomic length
$\sL_{\Lambda(\boldsymbol{s})}$ can also be expressed just by considering the
half square of the usual Euclidean norm evaluated on a particular lattice depending
on $n$ and $\ell$. Resume the notation of the previous paragraph and set
\[
Q_{\boldsymbol{s}}(\boldsymbol{t})=\frac{n}{\ell}P_{\boldsymbol{s}}
(\boldsymbol{t})=\frac{n^{2}}{2\ell^{2}}\sum_{i=1}^{n}t_{i}^{2}-\frac{n}{\ell}
\sum_{i=1}^{n}(i-1)t_{i}-\frac{n}{\ell}c_{\boldsymbol{s}}.
\]
Clearly, $Q_{\boldsymbol{s}}(\boldsymbol{t})$ belongs to $\frac{n}
{\ell}\mathbb{Z}$ for any $\boldsymbol{t}\in D_{\boldsymbol{s}}$ and
$P_{\boldsymbol{s}}$ is universal on $D_{\boldsymbol{s}}$ if and only if
$Q_{\boldsymbol{s}}$ represents all the rationals in $\frac{n}{_{\ell}}\mathbb{N}$
when $\boldsymbol{t}$ runs over $D_{\boldsymbol{s}}$. We have moreover
\[
Q_{\boldsymbol{s}}(\boldsymbol{t})=\frac{n}{\ell}P_{\boldsymbol{s}}
(\boldsymbol{t})=\frac{1}{2}\sum_{i=1}^{n}\left(  \frac{nt_{i}}{\ell}\right)
^{2}-\sum_{i=1}^{n}(i-1)\left(  \frac{nt_{i}}{\ell}\right)  -\frac{n}
{\ell}c_{\boldsymbol{s}}
\]
which suggests to set $y_{i}=\frac{nt_{i}}{\ell}$ and works in $\frac{n}
{\ell}D_{\boldsymbol{s}}$.\ Therefore $P_{\boldsymbol{s}}$ is universal on
$D_{\boldsymbol{s}}$ if and only if
\[
F_{\boldsymbol{s}}(y)=\frac{1}{2}\sum_{i=1}^{n}y_{i}^{2}-\sum_{i=1}
^{n}(i-1)y_{i}-d_{\boldsymbol{s}}
\]
with
\[
d_{\boldsymbol{s}}=\frac{n^{2}}{2\ell^{2}}\sum_{i=1}^{n}(s_{i}^{^{\prime}}
)^{2}-\frac{n}{\ell}\sum_{i=1}^{n}(i-1)s_{i}^{\prime}
\]
is $\frac{n}{\ell}$-universal on $\frac{n}{\ell}D_{\boldsymbol{s}}$ (i.e. attains all the rationals in $\frac{n}{\ell}\mathbb{N}$). In particular
$F_{\boldsymbol{s}}\left(  \frac{n}{\ell}s_{1}^{\prime},\ldots,\frac{n}{\ell}
s_{n}^{\prime}\right)  =F_{\boldsymbol{s}}\left(  \frac{n}{\ell}\boldsymbol{s}
^{\prime}\right)  =0$. Now observe that
\[
F_{\boldsymbol{s}}(y)=\frac{1}{2}\sum_{i=1}^{n}(y_{i}-i+1)^{2}-\frac
{n(n-1)(2n-1)}{12}-d_{\boldsymbol{s}}
\]
so by setting $z_{i}=y_{i}-n+i$, we are reduced to the study of
\[
\hat{F}_{\boldsymbol{s}}(z)=\frac{1}{2}\left\Vert z\right\Vert ^{2}
-\widehat{d}_{\boldsymbol{s}}
\]
on
\[
\frac{n}{\ell}D_{\boldsymbol{s}}^{n}-\delta\text{ where }\delta=(0,1,\ldots,n-1)
\]
with
\[
\widehat{d}_{\boldsymbol{s}}=\frac{1}{2}\left\Vert \frac{n}{\ell}\boldsymbol{s}
^{\prime}-\delta\right\Vert ^{2}.
\]

Therefore, we get the following proposition, which reduces the study of the universality of $\sL_{\Lambda(\boldsymbol{s})}$ to a similar problem for the Euclidean norm on a subset of $\mathbb{Z}^n$ in the spirit of the questions evoked in the introduction.

\begin{Prop}
\label{Prop_DilLattice}
The atomic length $\sL_{\Lambda(\boldsymbol{s})}$ is universal if and only if
$\hat{F}_{\boldsymbol{s}}(z)=\frac{1}{2}\left\Vert z\right\Vert ^{2}
-\widehat{d}_{\boldsymbol{s}}$ is $\frac{n}{\ell}$-universal on $\frac{n}{\ell}D_{\boldsymbol{s}
}^{n}-\delta$.
\end{Prop}

\begin{Rem}
\ 

\begin{enumerate}

\item When $\ell=1$ and $\boldsymbol{s}=\boldsymbol{s}^{\prime}=(0)$,
$D_{\boldsymbol{s}}^{n}$ reduces to the subset $n\mathcal{Q}_{n}$ of
$\mathbb{Z}^{n}$ of vectors with coordinates multiple of $n$ and summing up to
zero.\ We have $\widehat{d}_{\boldsymbol{s}}=0$.\ We so get
\[
\hat{F}_{\boldsymbol{s}}(z)=\frac{1}{2}\left\Vert z\right\Vert ^{2}
-\frac{n(n-1)(2n-1)}{12}
\]
and the Granville-Ono theorem means that $\hat{F}_{\boldsymbol{s}}$ is
universal on $n\mathcal{Q}_{n}$ for $n\geq4$.

\item When $\ell=n$ and $\boldsymbol{s}=(0,\ldots,n-1)=\delta$, we get
$\boldsymbol{s}^{\prime}=\boldsymbol{s}=\delta$ and $\widehat{d}
_{\boldsymbol{s}}=0$.\ This gives
\[
\hat{F}_{\boldsymbol{s}}(z)=\frac{1}{2}\left\Vert z\right\Vert ^{2}
\]
where $D_{\boldsymbol{s}}^{n}=\Delta_{n}$ according to the results of
 \Cref{subsec_eucli_290}.

\item The previous proposition connects the universality of the atomic length to
that of the Euclidean norm on subsets of a $\mathbb{Z}$-lattice, say $L$,
which are essentially obtained from $L$ by removing some affine
hyperplanes.\ So, roughly speaking, in the problems considered in this paper,
the number of points to remove in our $\mathbb{Z}$-lattices $L$ located on a sphere of radius $k$ is
at most polynomial in $n$.\ On the other hand, the number of integral points on a sphere
of radius $k$ has rather an exponential growth of order $k^{(n/2)-1}$ (see for example
Chapter 13 in \cite{grosswald1985representations} ).\ Heuristically, one can thus expect a positive answer to many of our
universality problems when $n$ or the integers to represent become sufficiently large. This observation
will be in particular illustrated in Section \ref{sec_large_rank}.

\end{enumerate}
\end{Rem}

\subsection{\texorpdfstring{Sum of the first $\ell$ fundamental weights}{Sum of the first ℓ fundamental weights}} \label{sec_conj_trunc}

As already observed in \cite{CG2022}, the atomic length can only be universal
for dominant weights $\Lambda$ having at least a coordinate equal to $1$ in
their decomposition on the basis of fundamental weights (otherwise $1$ is not
in the image of $\sL_{\Lambda}$). Recall that $\ell$ is such
that $1\leq \ell\leq n$. Thus, we can reduce the study of the
universality $\sL_{\Lambda}$ to the elements $\boldsymbol{s}$ in
$\mathbb{Z}^{\ell}$ having at least a component appearing with multiplicity
$1$.\ In the rest of this paragraph, we will restrict ourselves to the
element $\boldsymbol{s}=(0,1,\ldots,\ell-1)$ defining the dominant weight
\[
\Lambda(\boldsymbol{s})=\Lambda_{0}+\cdots+\Lambda_{\ell-1}\text{.}
\]
Then, we get $\kappa^{\prime}=(\ell-1,\ldots,1,0,\ldots,0)\in\mathbb{Z}^{n}$ with
$\left\vert \boldsymbol{s}\right\vert =\frac{\ell(\ell-1)}{2}$ that is
\[
(s_{1}^{\prime},\ldots,s_{n}^{\prime})=(0^{n-\ell+1},1,\ldots,\ell-1).
\]
We obtain
\[
\sL_{\Lambda_{0}+\cdots+\Lambda_{\ell-1}}(w)=P_{(0,\ldots,\ell-1)}(\boldsymbol{t}
)=\frac{n}{2\ell}\sum_{i=1}^{n}t_{i}^{2}-\sum_{i=1}^{n}(i-1)t_{i}-c_{(0,\ldots
,\ell-1)}
\]
with

\[
c_{(0,\ldots,\ell-1)}=\frac{n}{2\ell}\sum_{i=1}^{n}(s_{i}^{^{\prime}})^{2}
-\sum_{i=1}^{n}(i-1)s_{i}^{\prime}=\frac{n}{2\ell}\sum_{i=1}^{\ell-1}i^{2}
-\sum_{i=1}^{\ell-1}(n-l+i)i=\frac{(\ell-1)}{12}\left(  2\ell^{2}-4nl+2\ell-n\right)
\]

\begin{Exa}
Assume $\ell=2$.\ We have $\boldsymbol{s}=(0,1),$ $\boldsymbol{s}^{\prime
}=(0^{n-1},1),c_{\boldsymbol{s}}=1-\frac{3}{4}n$ and
\[
P_{(0,1)}(\boldsymbol{t})=\frac{n}{4}\sum_{i=1}^{n}t_{i}^{2}-\sum_{i=1}
^{n}(i-1)t_{i}+\frac{3}{4}n-1
\]
where $\boldsymbol{t}$ runs over the $n$-tuples of integers summing up $1$
which are all even but one. We can set $\boldsymbol{t}=2\boldsymbol{u}
+\varepsilon_{k}$ where $\varepsilon_{k}$ is one of the vectors of the
canonical basis of $\mathbb{Z}^{n}$ and $\sum_{i=1}^{n}u_{i}=0$. By an easy
computation, one obtains
\[
P_{\boldsymbol{s}}(\boldsymbol{t})=2\left(  \frac{n}{2}\sum_{i=1}^{n}u_{i}
^{2}-\sum_{i=1}^{n}(i-1)u_{i}\right)  +n-k+nu_{k}
\]
and one recognizes an expression of the atomic length of level $1$. Namely, we
have
\[
P_{(0,1)}(\boldsymbol{t})=2P_{0}(u)+n-k+u_{k}\text{ where }\boldsymbol{t}
=2\boldsymbol{u}+\varepsilon_{k}.
\]
Finding the generating series of $P_{\boldsymbol{s}}$ in this case is an
interesting non trivial problem.\ 
Also, it is not clear whether
the universality of $P_{\boldsymbol{s}}$ can be studied
using that of $P_{\boldsymbol{0}}$, notably because $k$ can be any integer between $1$ and $n$.
\end{Exa}

\begin{Conj}
\label{conj_trunc_affine}
Assume $n\geq5$. Then, for any $1\leq \ell\leq n$, the atomic length
$\sL_{\Lambda_{0}+\cdots+\Lambda_{\ell-1}}$ is universal. Equivalently, the
polynomial $P_{(0,1,\ldots,\ell-1)}$ is universal on $D_{(0^{n-\ell+1}
,1,\ldots,\ell-1)}$.
\end{Conj}

In addition to computational evidence, 
this conjecture is supported by the results of \Cref{sec_finite}, 
where we prove that the finite analogue statement holds, and in fact for all classical types.
In general, characterising universal weights remains a complicated open problem.

\subsection{A refinement of the Granville-Ono problem}

Let us recall the notation of \Cref{subsec_ns_core}. We shall assume that $n=\ell$ in the paragraph. 
For abaci with only one runner, it was
observed that the bijection $\varphi$ coincides with the usual notion of
$n$-quotient of the charged partition $(\lambda,s)$.\ In particular, 
we recover the bijection between $n$-cores partitions with charge $s$ and
elements of $\mathbb{Z}^{n}[s]$, that is $n$-tuples of integers summing up to $s$. 
On the other hand, we described at the beginning of \Cref{Subsec_AtomLengthGeneralDom} a one-to-one correspondence between
$(n,\boldsymbol{s})$-cores and the subset of $\mathbb{Z}^{n}[s]$ of $n$-tuples
whose distribution modulo $n$ coincides with that of $\boldsymbol{s}^{\prime}
$. This therefore gives us an embedding of the set of $(n,\boldsymbol{s}
)$-cores into the set of (usual) $n$-cores.\ We can then refine the
Granville-Ono problem by asking whether the subset of $n$-cores obtained this way 
is still universal for a fixed core multicharge $\boldsymbol{s}$. The goal of this
paragraph is to explain how this question can also be reduced to the study of
the Garvan-Kim-Stanton polynomial \Cref{GKS_pol},
or its slight generalisation given in \Cref{exa_Jac_GKS},
evaluated on the orbit of $\boldsymbol{s}^{\prime}$ under the action of the affine group $W$.

\medskip

So fix a $n$-tuple $\boldsymbol{s}=(s_{1},\ldots,s_{n})$ such that $0\leq
s_{1}\leq\cdots\leq s_{n}<n$ with $s=\sum_{i=1}^{n}s_{i}$ and consider the
quadratic polynomial
\[
P_{\boldsymbol{s}}(t_{1},\ldots,t_{n})=\frac{n}{2}\sum_{i=1}^{n}t_{i}^{2}
+\sum_{i=1}^{n}(i-1)t_{i}-\frac{s}{2}(n-1)-\frac{s^{2}}{2}
\]
evaluated on the set $\mathcal{O}_{\boldsymbol{s}}$ of $n$-tuples
$\boldsymbol{t}=(t_{1},\ldots,t_{n})$ such that $\sum_{i=1}^{n}t_{i}=s$ and
$\{t_{1}\mod n,\ldots,t_{n}\mod n\}=\{s_{1}^{\prime
},\ldots,s_{n}^{\prime}\}$ (i.e., the distribution of the residues in the
$n$-tuples $\boldsymbol{s}^{\prime}$ and $\boldsymbol{t}$ is the same). By
results evoked in \Cref{subsec_ns_core},
$P_{\boldsymbol{s}}(t_{1},\ldots,t_{n})$ counts the size of the
$n$-core with charge $s$ whose $n$-quotient corresponds to
$\boldsymbol{t}$.\ Moreover $\mathcal{O}_{\boldsymbol{s}}$ then corresponds to
the orbit of $\boldsymbol{s}$ under the action of the affine symmetric group
such that
\begin{gather*}
\mathtt{s}_{i}(t_{1},\ldots,t_{i},t_{i+1},\ldots,t_{n})=(t_{1},\ldots
,t_{i+1},t_{i},\ldots,t_{n})\text{ for }i=1,\ldots n\text{ and}\\
\mathtt{s}_{0}(t_{1},t_{2}\ldots,t_{n-1},t_{n})=(t_{n}-n,t_{2}\ldots
\ldots,t_{n-1},t_{1}+n)
\end{gather*}

\medskip

Assume that $\boldsymbol{s}=\boldsymbol{s}^{\prime}=(0,\ldots,n-1)$.
Then $\mathcal{O}_{\bs}$ is the set of $n$-tuples $(t_{1},\ldots,t_{n})\in\Z^n$ with
distinct residues modulo $n$ and such that $t_{1}+\cdots+t_{n}=\frac{n(n-1)}{2}$.
Then, we get
\[
P_{\boldsymbol{s}}(0,\ldots,n-1)=\binom{n+2}{4}.
\]
In general
\[
P_{\boldsymbol{s}}(t_{1},\ldots,t_{n})=\frac{n}{2}\sum_{i=1}^{n}t_{i}^{2}
+\sum_{i=1}^{n}(i-1)t_{i}-\frac{n(n-1)^{2}(n-2)}{8}
\]
when $\boldsymbol{t}$ belongs to $\mathcal{O}_{\boldsymbol{s}}$. Here,
$P_{\boldsymbol{s}}(t_{1},\ldots,t_{n})$ is not symmetric in $t_{1}
,\ldots,t_{n}$ therefore we cannot assume $t_{i}=i\mod n$ for
$i=1,\ldots,n$.\ 

Now set
\begin{align*}
Q_{\boldsymbol{s}}(t_{1},\ldots,t_{n})  &  =P_{\boldsymbol{s}}(t_{1}
,\ldots,t_{n})-P_{\boldsymbol{s}}(0,\ldots,n-1)\\
& =\frac{n}{2}\sum_{i=1}^{n}t_{i}^{2}+\sum_{i=1}^{n}(i-1)t_{i}-\frac{n\left(
n-1\right)  \left(  n-2\right)  ^{2}}{12}.
\end{align*}
The following conjecture is supported by computational experiments.

\begin{Conj}\label{conj_refined_GO}
Assume $n\geq6$. Then $Q_{\boldsymbol{s}} (t_{1},\ldots,t_{n})$ is universal
on $\mathcal{O}_{\bs}$.
\end{Conj}

\begin{Rem}
For $n=5$, we conjecture that only $125$ is not represented by $Q_{\boldsymbol{s}}$.
\end{Rem}

\section{Universality in finite types}
\label{sec_finite}

In \Cref{sec_conj_trunc}, we considered the level $\ell$ weights 
$\La_0+\cdots+\La_{\ell-1}$ with $\ell\leq n$, and conjectured that the corresponding atomic lengths are universal.
We will now prove that the analogous statement holds in finite type, 
thereby providing further evidence to support \Cref{conj_trunc_affine}.
In fact, the results of this section will be established not only in type $A$, but for all finite classical types.

\medskip

Let us consider a finite root system of type $A_{n}, B_n, C_n$ or  $D_n$.
By analogy\footnote{
We will follow the conventions of \cite{BOURB} (and so will we in \Cref{sec_univB} and \Cref{sec_large_rank}), namely that $n$ is the special node of the Dynkin diagram.
Therefore, to define $\overline{\rho}_\ell$, we have to start summing from $\om_n$.} 
with the affine weight considered in \Cref{sec_conj_trunc},
we consider the weight $\overline{\rho}_\ell = \om_n+\om_{n-1}\cdots+ \om_{n-\ell+1}$ for $1\leq \ell\leq n$ (or $2\leq \ell\leq n$ in type $D_n$).
In particular, for $\ell =n$, we recover the dominant weight $\overline{\rho}:=\om_1+\cdots+\om_{n}$
which is the half-sum of the positive roots.

We will completely characterise in \Cref{thm_sat_trunc} when the atomic length $\sL_{\overline{\rho}_\ell}: {\mathring{W}} \to \N$ 
associated to $\overline\rho_\ell$ describes an interval:
this is the finite analogue of being universal.
First of all, let us compute the maximum of the function $\sL_{\overline{\rho}_\ell}$,
which we know is realised at the longest element $w_0$ of ${\mathring{W}}$, see \cite[Theorem 5.11]{CG2022}.
We will prove in \Cref{thm_sat_trunc} that $\rho_\ell$ is universal.

\begin{Lem}\label{lem_bornes}
Denote $b_{n,\ell} = \sL_{\overline{\rho}_\ell}(w_0)$ for all $1\leq \ell\leq n$.
Then
\begin{enumerate}
\item $b_{n,\ell} =\ds \frac{\ell(\ell+1)(3n-2\ell+2)}{6}$ in type $A_{n}$,
\item $b_{n,\ell} =\ds \frac{3n(n+1)(2\ell-1) - 2\ell(\ell^2-1)}{6}$ in type $B_n$,
\item $b_{n,\ell} =\ds \frac{(6n^2-1)\ell - \ell^2(2\ell-3)}{6}$ in type $C_n$,
\item $b_{n,\ell} =\ds \frac{(\ell-1)(3n^2-3n-\ell(\ell-2))}{3}$ in type $D_n$ (remember that $\ell\geq 2$ in this case).
\end{enumerate}
\end{Lem}

\begin{proof}
We have $\sL_{\overline{\rho}_\ell}(w_0)  = \h(\rho_\ell-w_0(\rho_\ell))$,
so let us express this quantity in the basis of the simple roots.
\begin{enumerate}
\item In type $A_n$, we have by \cite{BOURB},
\begin{equation}\label{om_A}
\begin{split}
\omega_i & = \frac{1}{n+1}[ 
(n-i+1)\al_1+2(n-i+1)\al_2+\cdots+ (i-1)(n-i+1)\al_{i-1} 
\\
& \hspace{5.5cm} + i (n-i+1) \al_i + i(n-i)\al_{i+1} + \cdots + i \al_n]
\end{split}
\end{equation}
Moreover, $w_0(\al_i) = -\al_{n+1-i}$, thus by \Cref{om_A},
$w_0(\om_i) = -\om_{n+1-i}$, and $w_0(\rho_\ell) = \om_1+\cdots+\om_\ell$.
But by \Cref{om_A} again, we see that
\begin{equation}\label{om_A_1}
\h(\om_i) = \frac{i(n-i+1)}{2} = \h(\om_{n-i+1}).
\end{equation}
Therefore,
\begin{align*}
\sL_{\overline{\rho}_\ell}(w_0) & = \h(\rho_\ell - w_0(\rho_\ell)) && 
\\
& = \h(\om_{n}+\cdots+\om_{n-r+1}) + \h(\om_1+\cdots +\om_\ell)  &&
\\
& =  2\h(\om_{1})+\cdots+ 2\h(\om_\ell) && \text{by \Cref{om_A_1}} 
\\
& = \sum_{i=1}^{\ell} 2\left( \frac{i(n-i+1)}{2} \right) && \text{by \Cref{om_A_1}} 
\\
& = (n+1)\sum_{i=1}^{\ell} i - \sum_{i=1}^{\ell} i^2 && 
\\
& = \frac{\ell(\ell+1)(3n-2\ell+2)}{6}  && \text{after simplification.} 
\end{align*}
\item In type $B_n$, $w_0=-1$, so 
$\rho_\ell - w_0(\rho_\ell) = 2\rho_\ell.$
Moreover, following \cite{BOURB},
\begin{equation}\label{om_B}
\begin{array}{ll}
\omega_i = \al_1+2\al_2+\cdots+ (i-1)\al_{i-1} + i (\al_i+\cdots + \al_n) &  \text{ if $i<n$, and}
\\
\om_n=\frac{1}{2}(\al_1+2\al_2\cdots+n\al_n) &
\end{array}
\end{equation}
Therefore, 
\begin{align*}
\sL_{\overline{\rho}_\ell}(w_0) & = \h(2\rho_\ell) = 2\h(\rho_\ell)
\\
& = 2\h(\om_{n-\ell+1}+\cdots+\om_{n-1}+\om_n) && \text{} 
\\
&=  2\sum_{i=1}^{\ell-1} \h(\om_{n-i})+ 2\h(\om_n)) && \text{} 
\\
& = 2\sum_{i=1}^{\ell-1} \left( \frac{(n-i)(n-i+1)}{2} + i(n-i)\right) + 2\frac{n(n+1)}{4} && \text{by \Cref{om_B}} 
\\
& = \frac{3n(n+1)(2\ell-1) - 2\ell(\ell^2-1)}{6}  && \text{after computation.} 
\end{align*}
\item In type $C_n$, we again have $w_0=-1$ and 
$\rho_\ell - w_0(\rho_\ell) = 2\rho_\ell.$
Following \cite{BOURB},
\begin{equation}\label{om_C}
\omega_i = \al_1+2\al_2+\cdots+ (i-1)\al_{i-1} + i (\al_i+\cdots + \frac{1}{2}\al_n)   \text{ for all $i\leq n$}
\end{equation}
Therefore, a similar computation yields
\begin{align*}
\sL_{\overline{\rho}_\ell}(w_0) 
& = 2\sum_{i=1}^{\ell-1} \left( \frac{(n-i)(n-i+1)}{2} + (i-\frac{1}{2})(n-i)\right) + 2\frac{n^2}{2} && \text{by \Cref{om_C}} 
\\
& = \frac{(6n^2-1)\ell - \ell^2(2\ell-3)}{6}  && \text{after computation.} 
\end{align*}
\item In type $D_n$, 
the behaviour of $w_0$ depends on the parity of $n$, see \cite{BOURB}, 
but for $\ell\geq 2$, we always have $w_0(\rho_\ell)=-\rho_\ell$ (which can be seen from \Cref{om_D} below), so  
$\rho_\ell - w_0(\rho_\ell) = 2\rho_\ell$ again.
Following \cite{BOURB},
\begin{equation}\label{om_D}
\begin{array}{ll}
\omega_i = \al_1+2\al_2+\cdots+ (i-1)\al_{i-1} + i (\al_i+\cdots + \al_{n-2} + \frac{1}{2}\al_{n-1} + \frac{1}{2}\al_n) &  \text{ if $i<n-1$,}
\\
\om_{n-1}= \frac{1}{2}\left( \al_1+2\al_2+\cdots+ (n-2) \al_{n-2} + \frac{n}{2}\al_{n-1} + \frac{n-2}{2}\al_n \right) &
\\
\om_{n}= \frac{1}{2}\left( \al_1+2\al_2+\cdots+ (n-2) \al_{n-2} + \frac{n-2}{2}\al_{n-1} + \frac{n}{2}\al_n \right) &
\end{array}
\end{equation}
A similar computation as before yields
\begin{align*}
\sL_{\overline{\rho}_\ell}(w_0) 
& = 2\sum_{i=1}^{\ell-1} \left( \frac{(n-i)(n-i+1)}{2} + (i-1)(n-i)\right) + 2\frac{n(n-1)}{2} && \text{by \Cref{om_D}} 
\\
& = \frac{(\ell-1)(3n^2-3n-\ell(\ell-2))}{3} && \text{after computation.} 
\end{align*}
\end{enumerate}
\end{proof}

\begin{Th}\label{thm_sat_trunc}
We have 
$$\sL_{\overline{\rho}_\ell}(W) = \llbracket  0, 
b_{n,\ell}
\rrbracket$$
if and only if
$n\neq 2$ and $\ell\leq n$, or $n = 2$ and $\ell \in \{1,3\}$.  
\end{Th}

\begin{proof}
First, note that if $\ell=n$, then $\sL_{\overline{\rho}_\ell}=\sL$
so we have
$\sL_{\overline{\rho}_\ell}(W)=\sL(W_n)=\llbracket  0,\sL(w_0) \rrbracket$ by \cite[Theorem 7.3]{CG2022}. 
So it only remains to consider the case $\ell<n$.

We proceed by induction on $n$, similarly to \cite[Proof of Theorem 7.3]{CG2022}.
Recall that we have computed $b_{n,\ell}=\sL_{\overline{\rho}_\ell}(w_0)$ in \Cref{lem_bornes}.
Let $n\geq 2$ and
assume the following induction hypothesis: $\sL_{\overline{\rho}_\ell}(W) = \llbracket  0, b_{n,\ell} \rrbracket$  for all $\ell<n$.
We want the rank to appear clearly in the notation (for the purpose of induction), so we write 
$\sL_{\overline{\rho}_\ell} : W_n\to\N$ for the atomic length in rank $n$.
Now, consider $\sL_{\overline{\rho}_\ell} : W_{n+1} \to \N$ with $\ell < n+1$.
In order to show that $\sL_{\overline{\rho}_\ell}(W_{n+1})= \llbracket  0, b_{n+1,\ell} \rrbracket$,
is suffices to show that 
\begin{equation}\label{eq1}
\left\llbracket  0, \frac{b_{n+1,\ell}}{2} \right\rrbracket \subseteq \sL_{\overline{\rho}_\ell}(W_{n+1})
\end{equation}
and to use  \cite[Theorem 5.10]{CG2022} to conclude.

In order to do that, let $I=\{2,\ldots, n+1\}$ and consider the corresponding standard parabolic subgroup $W_I=\langle s_2,\ldots, s_{n+1}\rangle \leq W_{n+1}$,
so that $W_I \simeq W_{n}$ 
Therefore, by induction hypothesis, we have
$$\sL_{\overline{\rho}_\ell} (W_I) = \llbracket 0,b_{n,\ell}\rrbracket.$$
This implies 
$$\llbracket  0, b_{n,\ell} \rrbracket = \sL_{\overline{\rho}_\ell}(W_I) \subseteq \sL_{\overline{\rho}_\ell}(W_{n+1}).$$
Therefore, it suffices to show that 
\begin{equation}\label{eq2}
b_{n,\ell}\geq  \frac{b_{n+1,\ell}}{2}
\end{equation}
to prove \Cref{eq1}. We prove \Cref{eq2} case by case by looking at the different Dynkin types and using
\Cref{lem_bornes}.
\begin{enumerate}
    \item In type $A_n$, we have
\begin{align*}
b_{n,\ell} -  \frac{b_{n+1,\ell}}{2} &= \frac{\ell(\ell+1)(3n-2\ell-1)}{12}
\\
& >  \frac{\ell(\ell+1)(\ell-1)}{12} && \text{since $\ell < n$}
\\
&\geq  0 && \text{since $\ell\geq 1$.}
\end{align*}
\item In type $B_n$, we have
\begin{align*}
b_{n,\ell} -  \frac{b_{n+1,\ell}}{2} &=
 \frac{  3(n+1)(2\ell-1)(n-2) - 2\ell(\ell^2-1) }{12}
\\
& > \frac{3(\ell+1)(2\ell-1)(\ell-2) - 2\ell(\ell^2-1) }{12} && \text{since $\ell < n$}
\\
& = \frac{(\ell+1)(4\ell^2-13\ell+6) }{12} && \text{}
\end{align*}
Since the largest root of $4\ell^2 - 13\ell+6$ is $\frac{13+\sqrt{73}}{8}\approx 2.69$,
we get  $b_{n,\ell} -  \frac{b_{n+1,\ell}}{2} \geq 0$ as soon as $\ell\geq 3$.

We treat the remaining cases separately. 
For $\ell=1$, we obtain
$b_{n,\ell} -  \frac{b_{n+1,\ell}}{2} = \frac{(n+1)(n-2)}{4} \geq 0$ since $n\geq 2$.
For $\ell=2$, we obtain 
$b_{n,\ell} -  \frac{b_{n+1,\ell}}{2} = \frac{3n^2-3n-10}{4} \geq 0$ as soon as $n\geq 3$
since the largest root of $3n^2-3n-10$ is $\frac{3+\sqrt{129}}{6}\approx 2.39$.
Note that the only remaining case $n=2$ and $\ell=2$ (i.e. $\rho_\ell=\rho$ in rank $2$) is not saturating \cite[Theorem 7.3]{CG2022}. 
\item In type $C_n$, we have
\begin{align*}
b_{n,\ell} -  \frac{b_{n+1,\ell}}{2} &=
 \frac{  \ell(6n^2-12n-2) - 2\ell(\ell^2-1) }{12}
\\
& > \frac{  \ell(6\ell^2-12\ell-2) - 2\ell(\ell^2-1) }{12} && \text{since $\ell < n$}
\\
& = \frac{\ell(4\ell^2-9\ell-2) }{12}. &&
\end{align*}
Since the largest root of $4\ell^2 - 9\ell -2$ is $\frac{9+\sqrt{97}}{8}\approx 2.36$,
we get  $b_{n,\ell} -  \frac{b_{n+1,\ell}}{2} \geq 0$ as soon as $\ell\geq 3$.

We treat the remaining cases separately. 
For $\ell=1$, we obtain
$b_{n,\ell} -  \frac{b_{n+1,\ell}}{2} = \frac{6n^2-12n-1}{12} \geq 0$ as soon as $n\geq 2$ (so always)
since the largest root of $6n^2-12n-1$ is $\frac{12+2\sqrt{42}}{12}\approx 2.08$.
For $\ell=2$, we obtain 
$b_{n,\ell} -  \frac{b_{n+1,\ell}}{2} = \frac{3n^2-8n-2}{3} \geq 0$ as soon as $n\geq 3$
since the largest root of $3n^2-3n-10$ is $\frac{8+2\sqrt{22}}{6}\approx 2.89$.
Again, the only remaining case $n=2$ and $\ell=2$ (i.e. $\rho_\ell=\rho$ in rank $2$) is not saturating \cite[Theorem 7.3]{CG2022}. 
\item In type $D_n$, we have
\begin{align*}
b_{n,\ell} -  \frac{b_{n+1,\ell}}{2} &=
\frac{  (\ell-1)\left( 3n(n-3) - \ell(\ell-2) \right) }{6}
\\
& > \frac{    (\ell-1)\left( 3\ell(\ell-3) - \ell(\ell-2) \right)   }{6} && \text{since $\ell < n$}
\\
& = \frac{ (\ell-1)\ell(2\ell-7) }{6} && \text{}
\\
& \geq 0 && \text{as soon as $\ell\geq 4$}.
\end{align*}

We treat the remaining cases separately. 
For $\ell=2$, we obtain
$b_{n,\ell} -  \frac{b_{n+1,\ell}}{2} = \frac{n(n-3)}{2} \geq 0$ since $n\geq 4$.
For $\ell=3$, we obtain 
$b_{n,\ell} -  \frac{b_{n+1,\ell}}{2} = n^2-3n-1 \geq 0$ as soon as $n\geq 4$ (so always)
since the largest root of $n^2-3n-1$ is $\frac{3+\sqrt{13}}{2}\approx 3.30$.
\end{enumerate}
\end{proof}

\section{\texorpdfstring{Entropy of permutations in affine type $C_{n}^{(1)}$}{Entropy of permutations in affine type Cn(1)}} \label{sec_univB}

\subsection{\texorpdfstring{The type $A_{2n}^{(1)}$ embedding}{The type A2n(1) embedding}}

\label{Subsec_typeBEmbedding}The goal of this section is to establish that
each nonnegative integer can also be regarded as the entropy of an affine
permutation of type $C_{n}^{(1)}$ once the associated Weyl group is realised
properly.\ The idea will be to identify this entropy as an atomic length for a
dominant weight in the affine root system of affine type $A_{2n}^{(2)}$. In
this section, we fix the integer $n$ and denote by $W_{C}$ and $W$ the affine
Weyl groups of type $C_{n}^{(1)}$ and $A_{2n}^{(1)}$, respectively. Then, it is classical to
realise $W_{C}$ as the subgroup of $W$ with generators
\[
\left\{
\begin{array}
[c]{l}
s_{n}^{C}=s_{n}\\
s_{i}^{C}=s_{i}s_{2n-i},i=1,\ldots,n-1\\
s_{0}^{C}=s_{2n}s_{0}s_{2n}
\end{array}
\right.
\]
where the $s_{i},i=0,\ldots,2n$ are the usual generators of $W$. It is then
easy to check that with this realisation, the elements of $W_{C}$ are
precisely those satisfying the relation
\[
w(2n+1)=2n+1\text{ and }w(2n+1-i)=2n+1-w(i)\text{ for any }i=1,\ldots,2n.
\]
and $w(2n+1)=2n+1$.\ Alternatively, the elements in
$W_{C}$ are exactly those in $W$ such that
\[
w(-i)=-w(i)\text{ for any }i\in\mathbb{Z}
\]
but the first characterization is easier to see in the window notation of the
elements of $W_{C}$.

We will now realize the entropy of the elements in $W_{C}$ as an atomic length
for the affine root system of type $A_{2n}^{(2)}$ with Dynkin diagram
\[
\overset{0}{\circ}\Longleftarrow\circ-\circ\cdots\circ-\circ\Longleftarrow
\overset{n}{\circ}.
\]
In fact, we will see this affine root system in that of type $A_{2n}^{(1)}$, as the subsystem of type
$A_{2n}^{(2)}$ with simple roots
\begin{equation}
\left\{
\begin{array}
[c]{l}
\alpha_{n}^{C}=2\alpha_{n}\\
\alpha_{i}^{C}=\alpha_{i}+\alpha_{2n-i},i=1,\ldots,n-1\\
\alpha_{0}^{C}=\alpha_{0}+\alpha_{2n}
\end{array}
\right.  \label{alpha_B}
\end{equation}
and fundamental dominant weights
\[
\left\{
\begin{array}
[c]{l}
\Lambda_{n}^{C}=2\Lambda_{n}\\
\Lambda_{i}^{C}=\Lambda_{i}+\Lambda_{2n-i},i=1,\ldots,n-1\\
\Lambda_{0}^{C}=\Lambda_{0}+\Lambda_{2n}
\end{array}
\right.
\]
Now consider the weight
\[
\pi=\sum_{i=1}^{n-1}2\Lambda_{i}^{C}+2\Lambda_{0}^{C}+\Lambda_{n}^{C}
=2\sum_{i=0}^{2n}\Lambda_{i}=2\rho\text{.}
\]
For any $w$ in $W_{C}$, we can decompose $\pi-w(\pi)$ on the basis $\alpha
_{i}^{C},i=0,\ldots,n$ on the form
\[
\pi-w(\pi)=\sum_{i=1}^{n-1}a_{i}^{C}\alpha_{i}^{C}+a_{0}^{C}\alpha_{0}
^{C}+a_{n}^{C}\alpha_{n}^{C}.
\]
Using (\ref{alpha_B}), we then obtain
\[
\pi-w(\pi)=\sum_{i=1}^{n-1}a_{i}^{C}(\alpha_{i}+\alpha_{2n-i})+2a_{n}
^{C}\alpha_{n}+a_{0}^{C}(\alpha_{0}+\alpha_{2n}).
\]
On the other hand, since $\pi=2\rho$, we also have by decomposing this time
$\rho-w(\rho)$ on the basis $\alpha_{i},i=0,\ldots,2n$
\[
\pi-w(\pi)=2(\rho-w(\rho))=\sum_{i=0}^{2n}2a_{i}\alpha_{i}.
\]
We thus get
\[
\left\{
\begin{array}
[c]{l}
a_{n}^{C}=a_{n}\\
a_{i}^{C}=2a_{i}=2a_{2n-i}=a_{i}+a_{2n-1},i=1,\ldots,n-1\\
a_{0}^{C}=2a_{0}=2a_{2n}=a_{0}+a_{2n}
\end{array}
\right.  .
\]
Hence, if we denote by $\sL_{\pi}^{C}$ the atomic length of type
$A_{2n}^{(2)}$ associated with the dominant weight $\pi$, we have for any $w$ in
$W_{C}$
\[
\sL_{\pi}^{C}(w)=\sum_{i=0}^{n}a_{i}^{C}=\sum_{i=0}^{2n}
a_{i}=\sL(w)=E(w).
\]
In other terms the values taken by $\sL_{\pi}^{C}$ on $W_{C}$ coincide
with the entropy of the elements of $W_{C}$ regarded as affine permutations of
type $A_{2n}^{(1)}$.


\subsection{\texorpdfstring{The problem of the universality of the entropy for affine type $C$ permutations}{The problem of the universality of the entropy for affine type C permutations}}

As explained in \Cref{Subsec_typeBEmbedding}, the values of any affine
permutation $w$ in $W_{C}$ are completely determined by the sequence
$(w(1),\ldots,w(n))$ such that
$$
\left\{
\begin{array}{l}
w(i)  \neq0\mod(2n+1)\text{ \quad for all }1\leq i\leq n,\\
w(i)   \neq w(j)\mod(2n+1)\mand w(i)\neq
-w(j)\mod(2n+1)\text{ \quad for all }1\leq i<j\leq n.
\end{array}
\right.
$$
Indeed, this is equivalent to
\[
w(i)\neq w(j)\mod (2n+1)\text{ for any }1\leq i<j\leq2n+1
\]
with $w(2n+1)=2n+1$ and $w(2n+1-i)=2n+1-w(i)$ for any $i=1,\ldots,2n$.
Observe also if we set as in \Cref{subsec_eucli_290}
\[
w(i)=x_{i}+i\text{ for }1\leq i\leq2n+1
\]
we have the equivalence
\begin{gather*}
w(2n+1-i)=2n+1-w(i)\text{ for any }i=1,\ldots,2n\\
\Updownarrow\\
x_{2n+1-i}=-x_{i}\text{ for any }i=1,\ldots,2n.
\end{gather*}
Since $x_{2n+1}=0$, we therefore get for $w$ in $W_{C}$
\[
\sL_{\pi}^{C}(w)=\sL(w)=\frac{1}{2}\sum_{i=1}^{2n+1}x_{i}
^{2}=\sum_{i=1}^{n}x_{i}^{2}
\]
where $(x_{1},\ldots,x_{n})$ runs over the set
\[
\Delta_{n}^{C}=\left\{  x=(x_{1},\ldots,x_{n})\in\mathbb{Z}\;\middle|\;\left\{
\begin{array}
[c]{l}
x_{i}+i\neq0\mod (2n+1)\text{ for }1\leq i\leq n\\
x_{i}+i\neq x_{j}+j\mod (2n+1)\text{ }1\leq i<j\leq n\\
x_{i}+i\neq-(x_{j}+j)\mod (2n+1)\text{ }1\leq i<j\leq n
\end{array}
\right.  \right\}  .
\]
Now it is classical by Lagrange's four-square theorem that the Euclidean form
$\sum_{i=1}^{n}x_{i}^{2}$ is universal on $\mathbb{Z}^{n}$ for any $n\geq4$.
We conjecture the following result which will be proved in the rest of this section when
$2n+1$ is a prime number.

\begin{Conj}
	\label{Conj_UniverB}For any integer $n\geq4$, the Euclidean form is universal
	on the subset $\Delta_{n}^{C}$ of $\mathbb{Z}^{n}$.
\end{Conj}

\subsection{Connection with an elementary problem in additive combinatorics}
\label{subsec_sumset_C}

In this paragraph, we follow the same line as in \Cref{sec_univA} and prove that the
previous conjecture is a consequence of a simple statement in additive
combinatorics.\ First observe that we have a natural action of the
hyperoctahedral group ${\mathring{W}_C}$ (the Coxeter group of type $C_{n}$)
on $V_{C}=\left(  \mathbb{Z}/(2n+1)\mathbb{Z}\right)  ^{n}$.\ The generators
$\mathtt{s}_{i},i=1,\ldots,n-1$ acts on $V_{C}$ by permuting the $i$-th and $(i+1)$-th
coordinates and the generators $\mathtt{s}_{n}$ by changing the sign of the $n$-th
coordinate. For any $x\in\mathbb{Z}^{n}$, write $\overline{x}$ for its image
modulo $n$, that is in $V_{C}$. Let us denote by $\mathcal{C}_n$ the orbit of
$\overline{a}_{0}$ with $a_{0}=(1,2,\ldots,n)$ under this action. We have in
particular
\[
\mathcal{C}_n+\mathcal{C}_n=\mathcal{C}_n-\mathcal{C}_n.
\]
Observe also that for any $x\in\mathbb{Z}^{n}$, we have the equivalence
\begin{equation}
x\in\Delta_{n}^{C}\Longleftrightarrow\overline{x}+\overline{a}_{0}\in
\mathcal{C}_n. \label{equiB}
\end{equation}
We now prove the following stronger version of Conjecture \ref{Conj_UniverB}.

\begin{Th}\label{thm_sumset_C}
	Fix $n\geq4$. The sumset equality
	\begin{equation}
	\mathcal{C}_n+\mathcal{C}_n=\mathcal{C}_n-\mathcal{C}_n=V_{\mathcal{C}_n} \label{SumsetB}
	\end{equation}
	implies Conjecture \ref{Conj_UniverB}.
\end{Th}

\begin{proof}
We proceed in two steps.
\begin{enumerate}
    \item  By using (\ref{equiB}), the Euclidean norm is universal on
$\Delta_{C}$ if and only if, for each integer $k$ there exists $x=(x_{1}
,\ldots,x_{n})\in\mathbb{Z}^{n}\mathbb{\ }$such that
\begin{equation}
k=\sum_{i=1}^{n}x_{i}^{2}=\left\Vert x\right\Vert ^{2}\text{ and }\overline
{x}+\overline{a}_{0}\in \mathcal{C}_n. \label{CB}
\end{equation}
For any $x\in\mathbb{Z}^{n}\mathcal{\ }$and any $w \in {\mathring{W}_C}$, we
have $\left\Vert w\cdot x\right\Vert ^{2}=\left\Vert x\right\Vert ^{2}$. Now
assume that $\mathcal{C}_n-\mathcal{C}_n=V_{C}$ and consider an integer $k\in\mathbb{Z}$. Since
$n\geq4$, the Euclidean norm is universal on $\mathbb{Z}^{n}\mathcal{\ }$and
there exists $x\in\mathbb{Z}^{n}\mathcal{\ }$such that $k=\left\Vert
x\right\Vert ^{2}$ and $\overline{x}\in V_{C}$. Also $\mathcal{C}_n$ coincides with the
orbit of $\overline{a}_{0}$ under the action of ${\mathring{W}_C}$. By our
assumption $V_{C}=\mathcal{C}_n-\mathcal{C}_n$, there thus exist $w_{k},w_{k}^{\prime}$ in
${\mathring{W}_C}$ such that $\overline{x}=w_{k}^{\prime}\cdot\overline{a}
_{0}-w_{k}^{-1}\cdot\overline{a}_{0}$.

\item Since we have $k=\left\Vert x\right\Vert $, we get $k=\left\Vert
u\cdot x\right\Vert ^{2}$ for any $u\in {\mathring{W}_C}$ and in particular
$k=\left\Vert w_{k}\cdot x\right\Vert ^{2}$ for the element $w_{k}$ of step
1.\ Therefore, to obtain the universality of the Euclidean norm on $\Delta
_{n}^{C}$, it suffices to show that
\[
\overline{w_{k}\cdot x}+\overline{a}_{0}=w_{k}\cdot\overline{x}+\overline
{a}_{0}\in \mathcal{C}_n.
\]
This means that, thanks to the action of ${\mathring{W}_C}$ on $\mathbb{Z}
^{n}$, it suffices to check that one can always replace the vector $x$ by the
vector $w_{k}\cdot x$ in its orbit under the action of ${\mathring{W}_C}$ so
that Condition (\ref{CB}) becomes satisfied. By definition, $\mathcal{C}_n$ is stable
under the action of ${\mathring{W}_C}$ on $V_{C}$.\ Therefore, we get the
equivalence
\[
w_{k}\cdot\overline{x}+\overline{a}_{0}\in \mathcal{C}_n\Longleftrightarrow\overline
{x}+w_{k}^{-1}\cdot\overline{a}_{0}\in \mathcal{C}_n.
\]
Also $\mathcal{C}_n$ coincides with the orbit of $\overline{a}_{0}$ under the action of
${\mathring{W}_C}$ so that $\overline{x}+w_{k}^{-1}\cdot\overline{a}_{0}\in \mathcal{C}_n$
if any only if there exists $w^{\prime} \in {\mathring{W}_C}$ such that
$\overline{x}=w^{\prime}\cdot\overline{a}_{0}-w_{k}^{-1}\cdot\overline{a}_{0}$
which is guaranteed by our hypothesis $V_{C}=\mathcal{C}_n-\mathcal{C}_n$ by choosing $w^{\prime
}=w_{k}$ as exposed in 1. In conclusion, the sumset equality $V_{C}=\mathcal{C}_n-\mathcal{C}_n$
implies that the Euclidean norm is universal on $\Delta_{n}^{C}$.
\end{enumerate}
\end{proof}

By \Cref{thm_sumset_C}, the universality of the entropy
on the affine Weyl group $W_{C}$ is a consequence of the following conjecture
(supported by computational experiments).

\begin{Conj}\label{conj_sumset_C}
	\label{ConjB-B}For any integer $n\geq1$, we have
	\[
	\mathcal{C}_n-\mathcal{C}_n=V_{C}=\left(  \mathbb{Z}/(2n+1)\mathbb{Z}\right)  ^{n}.
	\]
\end{Conj}

\begin{Rem}
Observe that the statement of \Cref{conj_sumset_C} becomes false if we replace
$\mathbb{Z}/(2n+1)\mathbb{Z}$ by $\mathbb{Z}/2n\mathbb{Z}$. For instance the
orbit of $a_{0}=(1,2)$ under the action of the signed permutations is
$\mathcal{C}_{n}^{\prime}=\{(1,2),(2,1),(3,2),(2,3)\}$ and we only have
$\mathcal{C}_{n}^{\prime}-\mathcal{C}_{n}^{\prime}\varsubsetneqq
(\mathbb{Z}/4\mathbb{Z)}^{2}$ because $(1,0)$ does not belong to
$\mathcal{C}_{n}^{\prime}-\mathcal{C}_{n}^{\prime}$.
\end{Rem}

\subsection{The prime case}

The goal of this paragraph is to prove Conjecture \ref{ConjB-B} in the case
where $2n+1=p$ is a prime number, by using Alon's combinatorial Nullstellensatz \cite{Alon}
that we now recall.

\begin{Th}
	Let $K$ be an arbitrary field and $P=P(X_{1},\ldots,X_{n})$ a polynomial in
	$K[X_{1},\ldots,X_{n}]$ with degree $d=\sum_{i=1}^{n}t_{i}$ where
	$t_{1},\ldots,t_{n}$ are nonnegative integers such that the coefficient of
	$\prod_{i=1}^{n}X_{i}^{t_{i}}$ is nonzero.\ Then, for any subsets
	$S_{1},\ldots,S_{n}$ of $K$ such that $t_{i}<\left\vert S_{i}\right\vert $,
	there exists at least an element $(x_{1},\ldots,x_{n})\in S_{1}\times
	\cdots\times S_{n}$ such that
	\[
	P(x_{1},\ldots,x_{n})\neq0.
	\]
	
\end{Th}

\medskip

Fix a prime number $p>2$, set $p=2n+1$ and resume the notation of the
previous paragraph. We shall assume that $K=\mathbb{Z}/p\mathbb{Z}$. Since the
stabiliser of $\overline{a}_{0}$ under the action of ${\mathring{W}_C}$ is
trivial, each element of the group ${\mathring{W}_C}$ can be encoded by a
sequence $[w(\overline{1}),\ldots,w(\overline{n})]\in(\mathbb{Z}
/p\mathbb{Z)}^{n}$ such that
\begin{align*}
w(\overline{i})  &  \neq0\text{ for any }i=1,\ldots,n,\\
w(\overline{i})  &  \neq w(\overline{j})\text{ and }w(\overline{i}
)\neq-w(\overline{j})\text{ for any }1\leq i<j\leq n.
\end{align*}
It is also classical to observe that the square of the Vandermonde polynomial
\[
\Delta(T_{1},\ldots,T_{n})^{2}=\prod_{1\leq i<j\leq n}(T_{i}-T_{j})^{2}\in
K[X_{1},\ldots,X_{n}]
\]
is of degree $n(n-1)$ and its coefficient in $T_{1}^{n-1}\cdots T_{n}^{n-1}$
is equal to $(-1)^{\frac{n(n-1)}{2}}n!.$ To see this, observe that
\[
\Delta(T_{1},\ldots,T_{n})=\sum_{\sigma\in\mathfrak{S}_{n}}\varepsilon
(\sigma)\prod_{i=1}^{n}T_{i}^{n-\sigma(i)}
\]
and therefore the contribution of the monomials $T_{1}^{n-1}\cdots T_{n}
^{n-1}$ are exactly the products of the form
\[
\varepsilon(\sigma)\prod_{i=1}^{n}T_{i}^{n-\sigma(i)}\times\varepsilon
(\sigma_{0}\sigma)\prod_{i=1}^{n}T_{i}^{n-\sigma_{0}\sigma(i)},\sigma
\in\mathfrak{S}_{n}
\]
where $\sigma_{0}(i)=n+1-i$ for any $i=1,\ldots,n$. Since $\varepsilon
(\sigma)\varepsilon(\sigma_{0}\sigma)=\sigma(\sigma_{0})=(-1)^{\frac
	{n(n-1)}{2}}$, we get the coefficient $(-1)^{\frac{n(n-1)}{2}}
n!\mod p$ by summing over $\mathfrak{S}_{n}$. \emph{Here it is
	crucial to notice that }$p$\emph{ does not divide }$n!$\emph{ so that
}$(-1)^{\frac{n(n-1)}{2}}n!\mod p$\emph{ is nonzero}.

\medskip

Now fix $(a_{1},\ldots,a_{n})$ in $(\mathbb{Z}/p\mathbb{Z)}^{n}$ and define
the polynomial
\[
P_{a}(X)=\prod_{1\leq i<j\leq n}(X_{i}^{2}-X_{j}^{2})\left(  (X_{i}-a_{i}
)^{2}-(X_{j}-a_{j})^{2}\right)  \in\mathbb{Z}/p\mathbb{Z}[X_{1},\ldots
,X_{n}].
\]
It sould be clear that its degree $d$ and the coefficient $c$ of
$X_{1}^{2(n-1)}\cdots X_{n}^{2(n-1)}$ in $P_{a}(X)$ are the same as those in
\[
\Delta(X_{1}^{2},\ldots,X_{n}^{2})^{2}=\prod_{1\leq i<j\leq n}(X_{i}^{2}
-X_{j}^{2})^{2}.
\]
This means we have $d=2n(n-1)$ with $c=(-1)^{\frac{n(n-1)}{2}}
n!\mod p$ which is nonzero.

\medskip

Assume now that we put $S_{i}=\mathbb{Z}/p\mathbb{Z}\setminus\{0,a_{i}\}$ and
$t_{i}=2(n-1)$ for any $i=1,\ldots,n$. Then $t_{1}+\cdots+t_{n}=2n(n-1)=d$
and
\[
2(n-1)=t_{i}<\left\vert S_{i}\right\vert =2n-1\text{ for any }i=1,\ldots,n.
\]
We can then apply the previous theorem which gives an $n$-tuple $(x_{1}
,\ldots,x_{n})$ in $S_{1}\times\cdots\times S_{n}$ such that $P_{a}
(x_{1},\ldots x_{n})\neq0.$ This means that this $n$-tuple satisfies
$x_{i}\neq0,x_{i}-a_{i}\neq0$ for any $i=1,\ldots,n$ and moreover
\[
\left\{
\begin{tabular}
[c]{l}
$x_{i}\neq x_{j},x_{i}\neq-x_{j}$\\
$x_{i}-a_{i}\neq x_{j}-a_{j},x_{i}-a_{i}\neq-(x_{j}-a_{j})$
\end{tabular}
\right.  \text{ for any }1\leq i<j\leq n.
\]
Therefore, by setting $w_{1}(i)=x_{i}$ and $w_{2}(i)=x_{i}-a_{i}$ for any
$i=1,\ldots,n$, we so get two elements $w_{1}$ and $w_{2}$ in $W$ such that
\[
a_{i}=w_{1}(i)-w_{2}(i)\text{ for any }i=1,\ldots,n.
\]
We have proved the following theorem.

\begin{Th}
	\label{Th_HallB}
	Assume $p=2n+1$ is a prime number. Then any element $a$ in $(\mathbb{Z}
	/p\mathbb{Z})^{n}$ can be written as the difference of two elements in the
	orbit $\mathcal{O}$ of $(1,2,\ldots,n)$ under the action of ${\mathring{W}_C}$ on
	$(\mathbb{Z}/p\mathbb{Z})^{n}$, that is
	\[
	\mathcal{O}-\mathcal{O}=(\mathbb{Z}/p\mathbb{Z})^{n}\text{.}
	\]
	
\end{Th}

By the discussion of \Cref{subsec_sumset_C}, the following is immediate.

\begin{Cor}
\label{Cor_UnivB}
Assume $2n+1\geq5$ is prime. Then the entropy is universal on affine permutations of type ${C}_n$.
\end{Cor}

\begin{Rem}
    It is a natural question to ask for a similar result for the affine Weyl
    groups of types $B_{n}^{(1)}$ and $D_{n}^{(1)}$.\ We can proceed as for the
    affine group of type $C_{n}^{(1)}$ and realise them in a Weyl group of type
    $A_{2n}^{(1)}$. In fact they are both subgroups of $W_{C}$ characterised by
    parity conditions.\ Unfortunately we did not find a natural way to encode
    these parity conditions by using zeros of a convenient multivariable
    polynomial similar to the previous polynomial $P_{a}(X)$ introduced in our
    application of Alon's Nullstellensatz theorem.\ We think that the entropy is
    yet universal on both Weyl groups of types types $B_{n}^{(1)}$ and
    $D_{n}^{(1)}$ when $n$ is larger than an optimal bound.\ As for the general
    case of type $C_{n}^{(1)}$, a proof would probably necessitate a
    generalisation of Hall theorem to each of the classical Weyl groups of type
    $B_{n},C_{n}$ and $D_{n}$.
\end{Rem}

\section{Universality of the atomic length in large rank}
\label{sec_large_rank}

The goal of this section is to study the universality of the atomic length for
any classical affine root system for the weight
\[
\rho^{\vee}=h\Lambda_{0}+\overline{\rho}^{\vee}
\]
where $h$ is the Coxeter number given in Table \ref{TableX} and
$\overline{\rho}^{\vee}$ is the sum of the fundamental weights of its
associated finite dual root system.\ We refer the reader to \cite{Carter2005} for a complete
review on affine root systems and only recall in the sequel the material that
we need. In particular the classical affine root systems fall into seven
families classified according to their Cartan matrix $A$. The dual of an
affine root system with Cartan matrix $A$ is just the affine root system with
Cartan matrix $^{t}A$. For simply laced root system (as for type $A_{n}^{(1)}$), the Cartan matrix is symmetric so that $\rho^{\vee}=\rho$ is the sum of
the fundamental affine dominant weights.\ In the rest of this section, we will
consider an affine root system $X_{n}^{(a)}$ of rank $n$ and $a=1$ (resp. $2$)
in the untwisted (resp. twisted) case.\ We will also assume that $X_{n}^{(a)}\neq A_{n}^{(1)}$ since the case of $A_{n}^{(1)}$ has already been
studied in Section \ref{sec_univA}. When the affine root system considered has an
underlying finite root system of type $C_{n}$, the weight $\rho^{\vee}$ does
not belong to the weight lattice because $\overline{\rho}^{\vee}$ lies in a
weight lattice of type $B_{n}\varsupsetneq C_{n}$.\ Nevertheless, we will see
that the atomic length associated to $\rho^{\vee}$ takes nonnegative integer
values except when $X_{n}^{(a)}=A_{2n}^{(2)}$ where it takes nonnegative half-integer values.\ When $X_{n}^{(a)}\neq
A_{2n}^{(2)}$ (resp. $X_{n}^{(a)}=A_{2n}^{(2)}$), we will establish that it is
universal on $\mathbb{N}$ (resp. $\frac{1}{2}\mathbb{N}$) when the rank $n$ is
greater or equal to an explicit lower bound.

\subsection{\texorpdfstring{The atomic length for $\rho^{\vee}$}{The atomic length for ρ∨}}

By analogy with Section \ref{sec_univA} where $\rho^{\vee}=\rho$ for the root system of
type $A_{n}^{(1)}$, we will denote by $\sL$ the atomic length
associated to $\rho^{\vee}$ for our affine root system of type $X_{n}^{(a)}$.
The underlying classical finite root system is of type $\mathring{X}_{n}
=B_{n},C_{n}$ or $D_{n}$ and realised in the Euclidean space 
$\mathbb{R}^{n}=
{\textstyle\bigoplus\limits_{i=1}^{n}}
\mathbb{R\varepsilon}_{i}$.\ The root lattice of type $\mathring{X}_{n}$
is the $\mathbb{Z}$-lattice generated by the set $S$ of positive roots with

\begin{equation}
S=\left\{
\begin{array}
[c]{l}
\{\alpha_{i}=\varepsilon_{i}-\varepsilon_{i+1}\mid1\leq i<n,\alpha
_{n}=\varepsilon_{n}\}\text{ in type }B_{n},\\
\{\alpha_{i}=\varepsilon_{i}-\varepsilon_{i+1}\mid1\leq i<n,\alpha
_{n}=2\varepsilon_{n}\}\text{ in type }C_{n},\\
\{\alpha_{i}=\varepsilon_{i}-\varepsilon_{i+1}\mid1\leq i<n,\alpha
_{n}=\varepsilon_{n-1}+\varepsilon_{n}\}\text{ in type }D_{n}.
\end{array}
\right.  \label{Simple}
\end{equation}

The affine Weyl group $W$ of $X_{n}^{(a)}$ is the semi-direct product
${\mathring{W}}\ltimes M^{}$ where $M^{}$ is a sub $\mathbb{Z}$-lattice
of the root lattice given in \Cref{table}.
\begin{table}[h!]
\[
\label{TableX}
\begin{array}{@{}l@{\hskip 20pt}l@{\hskip 20pt}l@{\hskip 20pt}l@{\hskip 20pt}l@{}}
\hline
X_{n}^{(a)} & \mathring{X}_{n} & M^{} & \Vert{x}\Vert^{2} & h \\
\hline
B_{n}^{(1)} & B_{n} &
{\textstyle\bigoplus\limits_{i=1}^{n-1}}\mathbb{Z}\alpha_{i}\oplus2\mathbb{Z}\alpha_{n} &
\Vert{x}\Vert_{2}^{2} & 2n \\
C_{n}^{(1)} & C_{n} &
{\textstyle\bigoplus\limits_{i=1}^{n-1}}2\mathbb{Z}\alpha_{i}\oplus\mathbb{Z}\alpha_{n} &
\frac{1}{2}\Vert{x}\Vert_{2}^{2} & 2n \\
D_{n}^{(1)} & D_{n} &
{\textstyle\bigoplus\limits_{i=1}^{n}}\mathbb{Z}\alpha_{i} &
\Vert{x}\Vert_{2}^{2} & 2n-2 \\
A_{2n-1}^{(2)} & C_{n} &
{\textstyle\bigoplus\limits_{i=1}^{n}}\mathbb{Z}\alpha_{i} &
\Vert{x}\Vert_{2}^{2} & 2n-1 \\
A_{2n}^{(2)} & C_{n} &
{\textstyle\bigoplus\limits_{i=1}^{n-1}}\mathbb{Z}\alpha_{i}\oplus\frac{1}{2}\mathbb{Z}\alpha_{n} &
\Vert{x}\Vert_{2}^{2} & 2n+1 \\
D_{n+1}^{(2)} & B_{n} &
{\textstyle\bigoplus\limits_{i=1}^{n}}\mathbb{Z}\alpha_{i} &
2\Vert{x}\Vert_{2}^{2} & n+1\\
\hline
\end{array}
\]
\caption{Root lattices in affine classical types}
\label{table}
\end{table}

In particular, each element $w$ in $W$ can be written uniquely on the form
$w=\overline{w}t_{{y}}=t_{{x}}\overline{w}$ where ${x}=\overline
{w}^{-1}({y})$, $\overline{w}$ belongs to the finite classical Weyl group
${\mathring{W}}$ and ${y}$ belongs to $M^{}$.\ It then follows from Lemma
8.1 in \cite{CG2022} that the atomic than $\sL$ takes the form

\begin{equation}
\sL(w)=\sL_{\overline{\rho}^{\vee}}(\overline{w})+\frac{h^{2}
}{2}\Vert{x}\Vert ^{2}-h\langle\overline{\rho}^{\vee}
,{x}-\overline{w}^{-1}({x})\rangle.\label{L(W)Gne}
\end{equation}

where the values of $h$ and the expression of $\Vert{x}\Vert
^{2}$ in terms of the usual Euclidean norm $\left\Vert \cdot\right\Vert_2 ^{2}$
are given in Table \ref{table}. Here $\sL_{\overline{\rho}^{\vee}}$
is the atomic length for the weight $\overline{\rho}^{\vee}$ in the finite
Weyl group ${\mathring{W}}$. In fact, we will only need the crucial property of
$\sL_{\overline{\rho}^{\vee}}$ asserting that $\sL
_{\overline{\rho}^{\vee}}({\mathring{W}})$ is an interval in $\mathbb{N}$. More
precisely by Lemma 6.1 in \cite{CG2022} we have

\begin{equation}
\sL_{\overline{\rho}^{\vee}}({\mathring{W}})=\llbracket0,b_{n}\rrbracket\text{ with
}b_{n}=\left\{
\begin{array}
[c]{l}
\frac{1}{6}n(n+1)(4n-1)\text{ when }\mathring{X}_{n}=B_{n}\text{ or }C_{n},\\
\frac{1}{3}n(n-1)(2n-1)\text{ when }\mathring{X}_{n}=D_{n}.
\end{array}
\right.  \label{RangeLAfinite}
\end{equation}

\subsection{\texorpdfstring{Universality of $\sL$ in large rank}{Universality of sL in large rank}}

We first need the following Lemma.

\begin{Lem}
	\label{Lem_UniNorm}Assume $n\geq4$.
	
	\begin{enumerate}
		\item When $X_{n}^{(a)}\neq A_{2n}^{(2)}$, the map $f:{x}\mapsto\frac
		{1}{2}\Vert{x}\Vert ^{2}$ takes values in $\mathbb{N}$ and is
		universal on $M^{}$, that is $f(M^{})=\mathbb{N}$.
		
		\item When $X_{n}^{(a)}=A_{2n}^{(2)}$, the map $f:{x}\mapsto\frac
		{1}{2}\Vert{x}\Vert ^{2}$ takes values in $\frac{1}{2}
		\mathbb{N}$ and is universal on $M^{}$, that is $f(M^{})=\frac{1}{2}\mathbb{N}$.
	\end{enumerate}
\end{Lem}

\begin{proof}
	According to the description of the simple roots in (\ref{Simple}), we get the following
	alternative description of the lattice $M^{}$
\renewcommand{\arraystretch}{1.3}
    \[
\begin{array}{@{}l@{\hskip 20pt}l@{\hskip 20pt}l@{\hskip 20pt}l@{}}
\hline
X_{n}^{(a)} & \mathring{X}_{n} & M & \frac{1}{2}\Vert {x} \Vert^{2} \\
\hline
B_{n}^{(1)} & B_{n} & \mathbb{Z}_{0}^{n} &
\frac{1}{2}\Vert {x} \Vert_{2}^{2} \\

C_{n}^{(1)} & C_{n} & (2\mathbb{Z})^{n} &
\frac{1}{4}\Vert {x} \Vert_{2}^{2} \\

D_{n}^{(1)} & D_{n} & \mathbb{Z}_{0}^{n} &
\frac{1}{2}\Vert {x} \Vert_{2}^{2} \\

A_{2n-1}^{(2)} & C_{n} & \mathbb{Z}_{0}^{n} &
\frac{1}{2}\Vert {x} \Vert_{2}^{2} \\

A_{2n}^{(2)} & C_{n} & \mathbb{Z}^{n} &
\frac{1}{2}\Vert {x} \Vert_{2}^{2} \\

D_{n+1}^{(2)} & B_{n} & \mathbb{Z}^{n} &
\Vert {x} \Vert_{2}^{2}
\\
\hline
\end{array}
\]
	where $\mathbb{Z}_{0}^{n}:={\{x=(x}_{1},\ldots,x_{n})\in\mathbb{Z}^{n}\mid
	{x}_{1}+\cdots+x_{n}=0\mod 2\}$. The case $X_{n}^{(a)}=D_{n+1}^{(2)}$ is an easy consequence of Lagrange's four-square
	theorem.\ When $X_{n}^{(a)}=C_{n}^{(1)}$, one can set ${x}=2{x}^{\prime}$
	with ${x}^{\prime}\in\mathbb{Z}^{n}$, get $\frac{1}{4}\left\Vert
	{x}\right\Vert _{2}^{2}=\left\Vert {x}^{\prime}\right\Vert _{2}^{2}$ and
	conclude similarly.\ In type $B_{n}^{(1)},D_{n}^{(1)}$ and $A_{2n-1}^{(2)}$,
	it suffices to observe that for any ${x}\in\mathbb{Z}^{n}$, we have the
	equivalence
	\[
	\Vert{x}\Vert _{2}^{2}=0\mod 2\Longleftrightarrow
	{x}_{1}+\cdots+x_{n}=0\mod 2
	\]
	to get the set equality $f(M^{})=\mathbb{N}$. Finally, the last case
	$X_{n}^{(a)}=A_{2n}^{(2)}$ is also a consequence of Lagrange's four-square
	theorem$.$
\end{proof}

\medskip

Let us now observe that when $\overline{w}({x})={x}$ (i.e. ${x}
\in\mathrm{Stab}({x})$), we get by (\ref{L(W)Gne})
\[
\sL(w)=\sL_{\overline{\rho}^{\vee}}(\overline{w})+\frac{h^{2}
}{2}\Vert{x}\Vert ^{2}.
\]
The finite Dynkin diagram $\mathring{X}_{n}$ (obtained by removing the zero
node in $X_{n}^{(a)}$) can be pictured on the form
\[
\overset{1}{\circ}-\overset{2}{\circ}\cdots\overset{n-5}{\circ}-\overset
{n-4}{\circ}-\left[  \mathring{X}_{4}\right]
\]
where $\left[  \mathring{X}_{4}\right]  $ is the sub-Dynkin diagram of type
$\mathring{X}_{4}$ obtained by keeping only the nodes labelled by
$n-3,n-2,n-1$ and $n$ in $\mathring{X}_{n}$.\ Then $V_{4}=
{\textstyle\bigoplus\limits_{i=0}^{3}}
\mathbb{Z\alpha}_{n-i}$ is the root lattice associated to $\left[
\mathring{X}_{4}\right]  $.\ It is easy to check that it is stabilised by
${\mathring{W}}^{(4)}$, the subgroup of ${\mathring{W}}$ fixing the lattice
$V_{4}$ which is a Weyl group of type $\mathring{X}_{n-4}$.

\begin{Exa}
	Assume $\mathring{X}_{n}$ is of type $C_{6}$.\ Then ${\mathring{W}}^{(4)}$ is
	the subgroup of ${\mathring{W}}$ fixing the four last coordinates of ${x}
	\in\mathbb{Z}^{6}$.\ This is a hyperoctahedral group of rank $2$ acting on
	$\mathbb{Z}^{6}$ by permuting or changing the signs of the two first coordinates.
\end{Exa}

For any ${x}$ in $V_{4}$ and any $\overline{w}$ in $\mathring
{W}^{(4)}$, we so get
\[
\sL(w)=\sL_{\overline{\rho}^{\vee}}(\overline{w})+\frac{h^{2}
}{2}\Vert{x}\Vert ^{2}.
\]
According to Lemma \ref{Lem_UniNorm}, when $X_{n}^{(a)}\neq A_{2n}^{(2)}$, we
know that $f:{x}\longmapsto\frac{1}{2}\Vert{x}\Vert ^{2}$ is
universal.\ Therefore, by considering the elements in $W$ of the form
$w=t_{{x}}\overline{w}$ with $(\nu,\overline{w})\in V_{4}\times\mathring
{W}^{(4)}$, we get the inclusion
\[
\sL(W)\supseteq
{\textstyle\bigcup\limits_{k\in\mathbb{N}}}
\llbracket  h^{2}k,h^{2}k+b_{n-4} \rrbracket.
\]
By using (\ref{RangeLAfinite}), we obtain
\[
b_{n-4}=\left\{
\begin{array}
[c]{l}
\frac{1}{6}(n-4)(n-3)(4n-17)\text{ when }\mathring{X}_{n-4}=B_{n-4}\text{ or
}C_{n-4}\\
\frac{1}{3}(n-4)(n-5)(2n-9)\text{ when }\mathring{X}_{n-4}=D_{n-4}
\end{array}
\right.
\]
and thanks to the values of $h$ in Table \ref{table}, we get
\[
h^{2}k+b_{n-4}\geq h^{2}(k+1)\Longleftrightarrow b_{n-4}\geq h^{2}
\Longleftrightarrow\left\{
\begin{array}
[c]{l}
n\geq15\text{ when }X_{n}^{(a)}=B_{n}^{(1)},C_{n}^{(1)},A_{2n-1}^{(2)}\\
n\geq16\text{ when }X_{n}^{(a)}=D_{n}^{(1)},\\
n\geq10\text{ when }X_{n}^{(a)}=D_{n+1}^{(2)}.
\end{array}
\right.
\]
When it happens, we thus obtain
\[
\sL(W)=\mathbb{N}\text{.}
\]
Now, when $X_{n}^{(a)}=A_{2n}^{(2)}$, we can argue similarly and get
\[
\sL(W)\supseteq
{\textstyle\bigcup\limits_{k\in\frac{1}{2}\mathbb{N}}}
\llbracket h^{2}k,h^{2}k+b_{n-4} \rrbracket.
\]
We have $h=2n+1$ and $b_{n-4}=\frac{1}{12}(n-4)(n-3)(4n-17)$ which gives
\[
b_{n-4}\geq h^{2}\Longleftrightarrow n\geq16.
\]

Let $n_0$ be the following integer, depending on the Dynkin type.
$$
\begin{array}{@{}l@{\hskip 20pt}@{}l@{\hskip 20pt}}
\hline
\text{Dynkin type} &  n_0 
\\
\hline
B_{n}^{(1)} & 15\\
C_{n}^{(1)} & 15\\
D_{n}^{(1)} & 16\\
A_{2n-1}^{(2)} & 15\\
A_{2n}^{(2)} & 16\\
D_{n+1}^{(2)} & 10\\
\hline
\end{array}
$$
We have proved the following theorem.

\begin{Th}
\label{Th_univlarge}
When $n\geq n_0$, the atomic length $\sL$ associated to the weight $\rho^{\vee}$ represents all nonnegative
integers for any classical root system $X_{n}^{(a)}\neq A_{2n}^{(2)}$ and all
nonnegative half-integers for $X_{n}^{(a)}=A_{2n}^{(2)}$.
\end{Th}

\begin{Rem}
    One can use the same method to reprove a weaker version of the universality of the atomic length $\sL$ in type $A_{n}^{(1)}$. Here the equality $b_{n-4}=\frac{1}{3}(n-4)(n-5)(n-3) \geq (n+1)^{2}=h^{2}$ is satisfies as soon as $n \geq 15$.
\end{Rem}

\section*{Acknowledgements}

We thank Emily Norton for many stimulating discussions.
The authors were supported by the Agence Nationale de la Recherche funding ANR CORTIPOM 21-CE40-001.

\bibliographystyle{alphaurl}
\bibliography{biblio}

\end{document}